\documentclass[12pt]{article}
\usepackage{amsmath}
\usepackage{amssymb}
\usepackage{amsthm}
\usepackage{amsfonts}
\usepackage{graphicx}
\usepackage{pstricks}

\newtheorem{thm}{Theorem}[section]
\newtheorem*{thmn}{Theorem}
\newtheorem*{mainthm}{Main Theorem}

\newtheorem{lemma}[thm]{Lemma}
\newtheorem{cor}[thm]{Corollary}
\theoremstyle{definition}
\newtheorem{defn}[thm]{Definition}
\newtheorem{example}[thm]{Example}

\theoremstyle{remark}
\newtheorem*{remark}{Remark}
\DeclareMathOperator{\Aut}{Aut} 
 \DeclareMathOperator{\Char}{Char}

\newcommand{\h}[1]{\mathbf{#1}}

\newcommand{\ul}[1]{\underline{#1}}
\newcommand{\m}[1]{\mathcal{#1}}

\newcommand{\F}{\mathbb F}

\title{Characteristic Subgroups of Finite Abelian Groups}
\author{Brent L. Kerby, Emma L. Turner \\ Brigham Young University}

\begin{document}
\bibliographystyle{plain}
\maketitle

\begin{abstract}
We consider the question: When do two finite abelian groups have isomorphic lattices of characteristic subgroups? An explicit description of the characteristic subgroups of such groups enables us to give a complete answer to this question in the case where at least one of the groups has odd order. An ``exceptional" isomorphism, which occurs between the lattice of characteristic subgroups of $Z_p\times Z_{p^2}\times Z_{p^4}$ and $Z_{p^2} \times Z_{p^5}$, for any prime $p$, is noteworthy.
\end{abstract}

In 1939, Baer \cite{baer_subabel} considered the question: When do two groups have isomorphic lattices of subgroups?  Since in general this is a very difficult problem, Baer restricted his attention primarily to the case of abelian groups. Even in this case, a complete solution has only very recently been obtained, in \cite{calug}. Most of the complications arise in the case where both groups are infinite of torsion-free rank 1. In particular, if both groups are finite, the situation is fairly uncomplicated; the following theorem, which provides a complete solution to the problem in this case, follows immediately from Theorem 1.1(b) of \cite{calug}, where the result is credited to Baer:

\begin{thmn}
Let $G$ and $H$ be two finite abelian groups. Then $G$ and $H$ have isomorphic lattices of subgroups if and only if there is a bijection $\phi$ from the set of Sylow subgroups of $G$ to the set of Sylow subgroups of $H$ such that for all Sylow subgroups $P$ of $G$,
\begin{enumerate}
\item[(i)] If $P$ is cyclic of order $p^n$ for prime $p$, then $\phi(P)$ is cyclic of order $q^n$ for some prime $q$.
\item[(ii)] If $P$ is not cyclic, then $\phi(P) \cong P$.
\end{enumerate}
\end{thmn}

Given this success, it seems natural to consider a related question: When do two groups have isomorphic lattices of characteristic subgroups? Again, the general problem seems to be very difficult. We will consider only the case of finite abelian groups. We show in \S4 that this problem can be reduced to the case in which both groups are abelian $p$-groups (for the same prime $p$). Our main result then gives a solution in the case $p\neq 2$:

\begin{mainthm}\label{mainthm}
Given a prime $p\neq 2$ and abelian $p$-groups
\begin{align*}
G=Z_{p^{\lambda_1}}\times Z_{p^{\lambda_2}}\times \cdots \times Z_{p^{\lambda_n}},& \quad
0 < \lambda_1 < \lambda_2 < \cdots < \lambda_n,
\\
H=Z_{p^{\mu_1}}\times Z_{p^{\mu_2}}\times \cdots \times Z_{p^{\mu_m}},&\quad
0 < \mu_1 < \mu_2 < \cdots < \mu_m, \quad m\leq n
\end{align*}
then $G$ and $H$ have isomorphic lattices of characteristic subgroups if and only if
\begin{enumerate}
\item[(i)] $G=H$, or
\item[(ii)] $G=Z_{p^k}\times Z_{p^{k+1}}$ and $H=Z_{p^{2k+1}}$ for some $k\in\mathbb{N}$, or
\item[(iii)] $G=Z_p\times Z_{p^2}\times Z_{p^4}$ and $H=Z_{p^2}\times Z_{p^5}$.
\end{enumerate}
\end{mainthm}

Theorem \ref{duplat} below shows that the restriction on $\lambda_i$ and $\mu_i$ in the Main Theorem (namely, that neither group may have repeated factors in its direct decomposition) is without loss of generality. The only remaining case then is to determine when two abelian 2-groups have isomorphic lattices of characteristic subgroups. The situation in this case is more complicated, and we have not yet been able to obtain a complete solution, although we are optimistic that one is attainable with further effort.

Given a group $G$, the \emph{automorphism classes} of $G$ are the orbits of $\Aut(G)$ acting on $G$ in the natural way. We will say that two elements of $G$ are \emph{automorphic} if they are in the same automorphism class. A characteristic subgroup of $G$ may then be defined as a subgroup which is a union of automorphism classes of $G$. We denote the lattice of characteristic subgroups of $G$ by $\Char(G)$. In \S1 and \S2, we give an explicit description of the automorphism classes and characteristic subgroups, respectively, of a finite abelian group $G$, as an understanding of these is prerequisite for approaching our main problem. These topics were considered already in 1905 and 1920 by G. A. Miller \cite{miller,miller2} and again, independently, in 1934 by Baer, who considered the more general case of periodic abelian groups \cite{baer}, and finally in 1935 by Birkhoff \cite{birkhoff_subabel}. We feel it is necessary, however, to give an independent treatment here for several reasons: first, in some cases we will need a more explicit description than has been given previously; second, in the earlier works some of the key proofs have been omitted or are incomplete, and this has led to some significantly erroneous claims (e.g., an error in \cite[p. 23]{miller} is discussed by Birkhoff in \cite[p. 393]{birkhoff_subabel}). Our method in \S1 and \S2 differs in several respects from earlier works, and many of our results here are new. We have identified those results which have appeared in earlier works, along with all the authors who proved or stated them previously.

In \S3 we collect some preliminary results on the lattice structure of $\Char(G)$, enabling us to prove our main result in \S4.

\ \\
Note: Most of the results in \S1-\S3 appeared as part of the Master's thesis of the first author \cite{kerby_masters}, under the supervision of Stephen P. Humphries.

\section{Automorphism Classes of Abelian Groups}

It is well known that any finite abelian group $G$ may be written as the direct product of its Sylow subgroups:
$$G = G_{p_1} \times G_{p_2} \times \cdots \times G_{p_n}.$$
Since the Sylow subgroups of an abelian group are characteristic, it follows that every automorphism $\phi \in \Aut(G)$ may be written
$$\phi=\phi_1\times\phi_2\times\cdots\times\phi_n, \text{ where $\phi_i\in\Aut(G_{p_i})$.}$$
From this it follows that the automorphism classes of $G$ are precisely the sets
$$O_1 \times O_2 \times \dots \times O_n, \text{ where $O_i$ is an automorphism class of $G_{p_i}$},$$
while the characteristic subgroups of $G$ are
$$H_1 \times H_2 \times \dots \times H_n, \text{ where $H_i$ is a characteristic subgroup of $G_{p_i}$}.$$
Using these facts, the problem of determining the automorphism classes and characteristic subgroups of $G$ is completely reduced to the case in which $G$ is a $p$-group. So for the remainder of this section and the next we will assume $G$ is a $p$-group.

Up to isomorphism, we may write $$G=Z_{p^{\lambda_1}}\times Z_{p^{\lambda_2}}\times \cdots \times Z_{p^{\lambda_n}},$$ where $\lambda_1 \leq \lambda_2 \leq \cdots \leq \lambda_n$. We define $\h \lambda(G)$ to be the tuple $(\lambda_1,\dots,\lambda_n)$ and by convention let $\lambda_0=0$. As we will be working extensively with such tuples of integers, it will be convenient to introduce some notation for dealing with them:

\begin{defn}
Given tuples $\h a=(a_1,\dots,a_n)$ and $\h b=(b_1,\dots,b_n)$ with integer entries, define
\begin{align*}
\h a \leq \h b &\text{ if $a_i \leq b_i$ for all $i\in\{1,\dots,n\}$};\\
\h a \wedge \h b &= (\min\{a_1,b_1\}, \min\{a_2,b_2\}, \dots, \min\{a_n,b_n\});\\
\h a \vee \h b &= (\max\{a_1,b_1\}, \max\{a_2,b_2\}, \dots, \max\{a_n,b_n\}).
\end{align*}
\end{defn}

Define $\Lambda(G)$ to be the set of tuples $$\Lambda(G)=\{\h a : \h 0 \leq \h a \leq \h \lambda(G)\}.$$ It is evident that $\Lambda(G)$, under the partial order $\leq$, forms a finite lattice in which $\wedge$ and $\vee$ are the greatest lower bound and least upper bound operators respectively. For $i\in\{1,\dots,n\}$, we define $\h e_i\in\Lambda(G)$ to be the tuple with zeros in each coordinate except with a 1 in the $i$th component.

Given a tuple $\h a \in \Lambda(G)$, we define $T(\h a)$ to be the set of elements $g\in G$ for which the $i$th component of $g$ has order $p^{a_i}$:
$$T(\h a)=\{(g_1,g_2,\dots,g_n) \in G : |g_i|=p^{a_i}\text{ for all $i=1,\dots,n$}\}.$$
Note that the sets $T(\h a)$ partition the group $G$. If $g \in T(\h a)$, we say that $T(\h a)$ is the \emph{type} of $g$.

\begin{lemma}\label{lemtype}
If $g,h \in G$ have the same type $T(\h a)$, then $g$ and $h$ are automorphic.
\end{lemma}
\begin{proof}
Write $g=(g_1,\dots,g_n)$ and $h=(h_1,\dots,h_n)$. Since $g$ and $h$ have the same type, we have $|g_i|=|h_i|$ for each $i\in\{1,\dots,n\}$.
It is well known that in a finite cyclic group if two elements have the same order then they are automorphic. So there are automorphisms $\phi_i \in \Aut(Z_{p^{\lambda_i}})$ with $\phi_i(g_i)=h_i$. Then $\phi=\phi_1\times\phi_2\times\cdots\times\phi_n$ is an automorphism of $G$ with $\phi(g)=h$.
\end{proof}

Lemma \ref{lemtype} says that each automorphism class of $G$ is a union of types. From this it follows that given two types $T(\h a)$ and $T(\h b)$, if some element of $T(\h a)$ is automorphic to some element of $T(\h b)$, then all elements of $T(\h a)$ are automorphic to all elements of $T(\h b)$, and we will say in this case that $T(\h a)$ and $T(\h b)$ are \emph{automorphic}.

\begin{defn}
Given a type $T(\h a)$, the automorphism class of $G$ containing $T(\h a)$ is denoted $O(\h a)$.
\end{defn}

\begin{defn}\label{defcan}
A type $T(\h a)$ is \emph{canonical} if
\begin{enumerate}
\item[(I)] $a_i \geq a_{i-1}$ for all $i\in\{2,\dots,n\}$, and
\item[(II)] $a_{i+1}-a_i \leq \lambda_{i+1}-\lambda_i$  for all $i\in\{1,\dots,n-1\}$.
\end{enumerate}
If (I) fails for a given $i$, we will say that $T(\h a)$ is \emph{type (I) noncanonical at} coordinate $i$, and similarly if (II) fails.
\end{defn}

Thus a type $T(\h a)$ is canonical if and only if $a_1,\dots,a_n$ is a (weakly) increasing sequence but at each step it increases by ``not too much", namely, by no more than the difference between the corresponding $\lambda$'s. In this case we will also say that the tuple $\h a$ itself is canonical. The set of canonical tuples will be denoted $\m C(G)$.

For what follows, it will be helpful to introduce some additional notation.
Let $t_1, t_2, \dots, t_n$ be generators for the respective cyclic factors in $G=Z_{p^{\lambda_1}}\times Z_{p^{\lambda_2}}\times \cdots \times Z_{p^{\lambda_n}}$. Also, for $0 \leq a \leq \lambda_i$, define $t_{i,a}={t_i}^{p^{\lambda_i-a}}$, so that $t_{i,a}$ is an element of order $p^a$ in $\langle t_i \rangle$, with $t_{i,\lambda_i}=t_i$.

The definition of ``canonical" is justified by the following theorem.

\begin{thm}\label{unican}
Every type is automorphic to a unique canonical type. Moreover, every canonical type is the maximum type in its automorphism class, i.e., if a type $T(\h a)$ is canonical then $\h a$ is the maximum element of $\{\h b : T(\h b) \subseteq O(\h a)\}$.
\end{thm}

Before proving this, we will need the following lemma.

\begin{lemma}\label{ncanup3}
Let $T(\h a)$ be noncanonical at coordinate $i$. Then $T(\h a+\h e_i)$ is automorphic to $T(\h a)$.
\end{lemma}
\begin{proof}
One element of type $T(\h a)$ is $g=\prod_{j=1}^n t_{j,a_j}$, so it is enough to show that there is an automorphism $\phi$ with $\phi(g) \in T(\h a+\h e_i)$.

First consider the case that $T(\h a)$ is type (I) noncanonical at $i$. So $a_i < a_{i-1}$. Define a homomorphism $\phi: G \to G$ by setting $\phi(t_j)=t_j$ for all $j\neq i-1$ and $\phi(t_{i-1})=t_{i-1}t_{i,s}$ where $s=\lambda_{i-1}-a_{i-1}+a_i+1$. This is well-defined since $|\phi(t_j)|$ divides $|t_j|$ for all $j$, because in fact equality holds: For $j \neq i-1$ this is trivial, while for $j=i-1$ we have $$|\phi(t_{i-1})|=|t_{i-1}t_{i,s}|=
\text{lcm}(|t_{i-1}|,|t_{i,s}|)=
\text{lcm}(p^{\lambda_{i-1}},p^s)=p^{\lambda_{i-1}}=|t_{i-1}|,$$
since $s \leq \lambda_{i-1}$. The image of $\phi$ contains each generator $t_j$ where $j\neq i-1$, and since $\phi(t_{i-1}t_{i,s}^{-1})=t_{i-1}$, the image of $\phi$ also contains $t_{i-1}$. Thus $\phi$ is onto, which, since $G$ is finite, implies $\phi$ is an automorphism. Now, \begin{align*}\phi(t_{i-1,a_{i-1}})&=\phi(t_{i-1}^{p^{\lambda_{i-1}-a_{i-1}}})=
(t_{i-1}t_{i,s})^{p^{\lambda_{i-1}-a_{i-1}}}\\
&=t_{i-1,a_{i-1}}t_{i,s-(\lambda_{i-1}-a_{i-1})}=t_{i-1,a_{i-1}}t_{i,a_i+1}.\end{align*}
Hence $$\phi(g)=t_{i,a_i+1}\prod_{j=1}^n t_{j,a_j}=t_{i,a_i}t_{i,a_i+1}\prod_{j\neq i} t_{j,a_j}.$$ Now, since $|t_{i,a_i}t_{i,a_i+1}|=|t_{i,a_i+1}|=p^{a_i+1}$, it follows that $\phi(g)$ has type $\h a+\h e_i$ as desired.

Now consider the case that $T(\h a)$ is type (II) noncanonical at $i$. So $a_{i+1}-a_i > \lambda_{i+1}-\lambda_i$. In this case, define $\phi$ by $\phi(t_j)=t_j$ for $j\neq i+1$ and $\phi(t_{i+1})=t_{i,s}t_{i+1}$ where $s=\lambda_{i+1}-a_{i+1}+a_i+1$. Again, this is well-defined since
$$|\phi(t_{i+1})|=|t_{i,s}t_{i+1}|=\text{lcm}(|t_{i,s}|,|t_{i+1}|)
=\text{lcm}(p^s,p^{\lambda_{i+1}})=p^{\lambda_{i+1}}=|t_{i+1}|,$$
since $s\leq \lambda_{i+1}$. Since $\phi$ is clearly surjective, it is an automorphism of $G$. We have
\begin{align*}\phi(t_{i+1,a_{i+1}})&=\phi(t_{i+1}^{p^{\lambda_{i+1}-a_{i+1}}})
=(t_{i,s}t_{i+1})^{p^{\lambda_{i+1}-a_{i+1}}}\\
&=t_{i,s-(\lambda_{i+1}-a_{i+1})}t_{i+1,a_{i+1}}
=t_{i,a_i+1}t_{i+1,a_{i+1}}.
\end{align*}
Hence, once more $$\phi(g)=t_{i,a_i+1}\prod_{j=1}^n t_{j,a_j},$$ so that again $\phi(g)$ has type $\h a+\h e_i$, as desired.
\end{proof}

\begin{proof}[Proof of Theorem \ref{unican}]
Let $T(\h a)$ be a type which is non-canonical. Lemma \ref{ncanup3} implies that $T(\h a)$ is automorphic to another type $T(\h a')$ where $\h a' > \h a$. If $\h a'$ is non-canonical, then we may again apply Lemma \ref{ncanup3} to obtain another automorphic type $T(\h a'')$ where $\h a'' > \h a'$. This process may be continued but must eventually terminate since there are no infinite increasing sequences in $\Lambda$. Hence $T(\h a)$ is automorphic to a canonical type $T(\h b)$. Since $\h b \geq \h a$, this also shows that $\h b$ is the maximum type in its automorphism class, assuming the uniqueness of $\h b$ which we now prove.

So let $\h a$ and $\h a'$ be distinct canonical types. We will show that $T(\h a)$ is not automorphic to $T(\h a')$. Let $g=\prod_{j=1}^n t_{j,a_j}$ and $g'=\prod_{j=1}^n t_{j,a_j'}$, so $g$ and $g'$ are elements of type $\h a$ and $\h a'$ respectively. Let $i$ be the least positive integer such that $a_i \neq a_i'$. Without loss of generality, assume $a_i < a_i'$.  Consider the elements $h=g^{p^{a_i}}$ and $h'=(g')^{p^{a_i}}$. Let $\h b$ and $\h b'$ be the types of $h$ and $h'$ respectively. By condition (I) of $\h a$ being canonical, we have $a_j \leq a_i$ for all $j < i$, hence $b_j=0$ for all $j \leq i$, while $b_j'=0$ for all $j < i$ but $b_i' \neq 0$. We have $b_j=a_j-a_i$ for all $j\geq i$. By condition (II) of $\h a$ being canonical, we have $\lambda_j-a_j \geq \lambda_i-a_i$ for all $j > i$, hence $\lambda_j-b_j \geq \lambda_i$ for all $j \geq i$. It follows that $h$ has a $p^{\lambda_i}$th root in $G$ while $h'$ does not, so $h$ and $h'$ are not automorphic. Consequently, $g$ and $g'$ cannot be automorphic, so $T(\h a)$ and $T(\h a')$ are not automorphic.
\end{proof}

We now obtain an important corollary, which was already discovered by Miller \cite[p. 23]{miller} and independently by Baer \cite[Corollary 2]{baer}, and proved again by Birkhoff in \cite[Theorem 9.4]{birkhoff}:

\begin{cor}[Miller-Baer-Birkhoff]
For any prime $p$, the number of automorphism classes of $Z_{p^{\lambda_1}}\times Z_{p^{\lambda_2}}\times \cdots \times Z_{p^{\lambda_n}}$ (where $\lambda_1 \leq \lambda_2 \leq \cdots \leq \lambda_n$) is $$\prod_{i=1}^n(\lambda_i-\lambda_{i-1}+1).$$
\end{cor}
\begin{remark}
This count includes the trivial automorphism class (containing only the identity element of the group), in spite of the curious statement to the contrary in \cite[p. 23]{miller}.
\end{remark}
\begin{proof}
Theorem \ref{unican} shows that the automorphism classes of $G$ are in one-to-one correspondence with the canonical tuples of $\Lambda(G)$. The canonical tuples $\h a$ are precisely those which satisfy $a_{i-1} \leq a_i \leq a_{i-1}+\lambda_i-\lambda_{i-1}$ for each $i \in \{1,\dots,n\}$. Thus there are $\lambda_i-\lambda_{i-1}+1$ choices for each coordinate $a_i$, and the result follows.
\end{proof}

\begin{example}
Let $G=Z_2\times Z_8=Z_2\times Z_{2^3}=\langle s\rangle\times\langle
t\rangle$. Then there are $(1-0+1)(3-1+1)=6$ automorphism classes of $G$, namely:
\begin{align*}
O(0,0)&=T(0,0)=\{1\},\\
O(0,1)&=T(0,1)=\{t^4\},\\
O(0,2)&=T(0,2)=\{t^2,t^6\},\\
O(1,1)&=T(1,1)\cup T(1,0)=\{s,st^4\},\\
O(1,2)&=T(1,2)=\{st^2,st^6\},\\
O(1,3)&=T(1,3)\cup T(0,3)=\{t,st,t^3,st^3,t^5,st^5,t^7,st^7\}.
\end{align*}
\end{example}

For information on how the automorphism classes split up as a union of types, see Theorem \ref{osplit} in the next section.

\section{Characteristic Subgroups of Abelian Groups}

We let $\Char(G)$ denote the lattice of characteristic subgroups of $G$.

\begin{defn}
Given an $n$-tuple $\h a \in \Lambda(G)$, we define the subgroup $R(\h a)=\underset{\h b\leq \h a}\cup T(\h b)$ and call $R(\h a)$ the \emph{regular subgroup below} $\h a$.
\end{defn}
\begin{remark}
We use the term ``regular", following Baer \cite{baer}. But this concept of regular should not be confused with the notion of a regular permutation group, nor of a regular $p$-group.
\end{remark}

\begin{thm}\label{canchar}
$R(\h a)$ is a characteristic subgroup if and only if $T(\h a)$ is a canonical type.
\end{thm}
\begin{proof}
Suppose first that $\h a$ is noncanonical. Then by Lemma \ref{ncanup3}, there is another tuple $\h a' > \h a$ with $O(\h a')=O(\h a)$. Then $R(\h a)$ contains $T(\h a)$ but not $T(\h a')$; this means that $R(\h a)$ contains some but not all of the automorphism class $O(\h a)$, so $R(\h a)$ is not characteristic.

Now assume $\h a$ is canonical. We need to show that $R(\h a)$ is a union of automorphism classes. Suppose by way of contradiction that there is a type $T(\h b)$ with $T(\h b) \subseteq R(\h a)$ but not $O(\h b) \subseteq R(\h a)$. Take $\h b$ to be a maximal such tuple. If $\h b$ is canonical, then for every type $T(\h c)$ contained in $O(\h b)$, we have $\h c \leq \h b$ since $\h b$ is the maximum type of its automorphism class by Theorem \ref{unican}. Hence $\h c \leq \h a$, so $T(\h c) \subseteq R(\h a)$. This implies $O(\h b) \subseteq R(\h a)$, contrary to assumption. So $\h b$ must be noncanonical. So there is some $i \in \{1,\dots,n-1\}$ such that either $b_{i+1} < b_i$ or $b_{i+1}-b_i > \lambda_{i+1}-\lambda_i$. In the first case, define $\h b'$ by $b_j'=b_j$ for $j\neq i+1$ and $b_{i+1}'=b_i$. By Lemma \ref{ncanup3}, $T(\h b)$ and $T(\h b')$ are automorphic types, i.e. $O(\h b)=O(\h b')$. Since $\h a$ is canonical, we have $a_i \leq a_{i+1}$, hence $b_{i+1}'=b_i \leq a_i \leq a_{i+1}$, so that $\h b' \leq \h a$. Then $T(\h b') \subseteq R(\h b)$ but not $O(\h b') \subseteq R(\h a)$. Since $\h b'>\h b$, this contradicts the maximality of $\h b$.

In the second case, i.e., if $b_{i+1}-b_i > \lambda_{i+1}-\lambda_i$, define $\h b'$ by $b_j'=b_j$ for $j\neq i$ and $b_i'=b_{i+1}-(\lambda_{i+1}-\lambda_i)$. Again by Lemma \ref{ncanup3}, $T(\h b)$ and $T(\h b')$ are automorphic types. Since $\h a$ is canonical, we have $a_{i+1}- (\lambda_{i+1}-\lambda_i) \leq a_i$. Hence $b_i' = b_{i+1}-(\lambda_{i+1}-\lambda_i) \leq a_{i+1}-(\lambda_{i+1}-\lambda_i) \leq a_i$, so $\h b' \leq \h a$. Then, as in the previous case, $T(\h b') \subseteq R(\h b)$ but not $O(\h b') \subseteq R(\h a)$, which contradicts the maximality of $\h b$, since $\h b'>\h b$,
\end{proof}

The following is easily verified by direct calculation:
\begin{thm}\label{Rlat}For any $\h a, \h b \in \Lambda(G)$,
\begin{enumerate}
\item[(i)] $R(\h a) \cap R(\h b) = R(\h a \wedge \h b)$;
\item[(ii)] $\langle R(\h a), R(\h b) \rangle = R(\h a \vee \h b)$;
\item[(iii)] $|R(\h a)|=p^{\sum_{i=1}^n a_i}$.
\end{enumerate}
\end{thm}
From (i) and (ii) and the fact that the meet and join of characteristic subgroups is characteristic, it follows that the regular characteristic subgroups form a sublattice of $\Char(G)$. Using Theorem \ref{canchar}, this then implies that if $\h a$ and $\h b$ are canonical tuples then so are $\h a \wedge \h b$ and $\h a \vee \h b$. (This is also not difficult to verify directly.)

The following theorem shows that irregular characteristic subgroups can only exist in the case $p=2$.
\begin{thm}[Miller-Baer]\label{charreg}
Let $G$ be an abelian $p$-group where $p \neq 2$. Then every characteristic subgroup of $G$ is regular.
\end{thm}
\begin{remark}
This theorem was shown by Baer in \cite[Theorem 9]{baer}. It was known to Miller although it is questionable whether his footnote in \cite[p. 21]{miller} constitutes a complete proof. A related result of Birkhoff is found in \cite[Theorem 10.1]{birkhoff}.
\end{remark}
\begin{proof}
Let $H$ be any characteristic subgroup of $G$. Define the $n$-tuple $\h m$ by $m_i=\max\{a_i : \h a \in \Lambda(G), T(\h a) \subseteq H\}$. It is clear then that $H \leq R(\h m)$. We will show that on the other hand $R(\h m) \leq H$, from which the result immediately follows.

For any $i$, by our definition of $\h m$ there is a type $\h a$ such that $T(\h a) \subseteq H$ and $a_i=m_i$. Then $g=\prod_{j=1}^n t_{j,a_j}$ and $g'=t_{i,m_i}\prod_{j\neq i} t_{j,a_j}^{-1}$ are two elements of $T(\h a)$. Since $H$ is a subgroup, $gg' = t_{i,m_i}^2 \in H$. Since $p \neq 2$, we have $\langle t_{i,m_i}^2 \rangle = \langle t_{i,m_i} \rangle$, so $t_{i,m_i} \in H$. Since the elements $t_{i,m_i}$ generate $R(\h m)$, it follows that $R(\h m) \leq H$, as desired.
\end{proof}

\begin{cor}\label{cordist}
Let $G$ be an abelian $p$-group where $p \neq 2$. Then the lattice $\Char(G)$ is isomorphic to the lattice $\m C(G)$. In particular, $\Char(G)$ is a distributive lattice.
\end{cor}
\begin{proof}
The first statement follows immediately from Theorem \ref{Rlat}(i,ii) and Theorem \ref{charreg}. The second statement holds since $\m C(G)$ is a sublattice of the lattice $\Lambda(G)$, which is distributive since it is a direct product of chains.
\end{proof}

Corollary \ref{cordist} enables us to give an explicit description of the lattice of characteristic subgroups of any abelian $p$-group of odd order. For example, $\Char(Z_p \times Z_{p^3})$, for an odd prime $p$, is shown in Table \ref{char13}.

\begin{table}
\centering
\caption{Characteristic subgroups of $G=Z_p\times Z_{p^3}$ for odd prime $p$}\label{char13}

\begin{tabular}{lr}
\begin{tabular}{|l|l|} \hline
$H_1$&R(0,0)\\ \hline
$H_2$&R(0,1)\\ \hline
$H_3$&R(0,2)\\ \hline
$H_4$&R(1,1)\\ \hline
$H_5$&R(1,2)\\ \hline
$H_6$&R(1,3)\\ \hline
\end{tabular}
&
\begin{tabular}{r}
\psset{xunit=.2cm,yunit=.2cm,labelsep=2.5mm}
\begin{pspicture*}(-10,-10)(10,22)
\Rput[t](0,0){$H_1$}

\Rput[t](0,5){$H_2$}
\psline(0,-.5)(0,1.5)

\Rput[t](-5,10){$H_3$}
\Rput[t](5,10){$H_4$}
\psline(-1.5,4.5)(-3.5,6.5)
\psline(1.5,4.5)(3.5,6.5)

\Rput[t](0,15){$H_5$}
\psline(-3.5,9.5)(-1.5,11.5)
\psline(3.5,9.5)(1.5,11.5)

\Rput[t](0,20){$H_6$}
\psline(0,14.5)(0,16.5)
\end{pspicture*}
\end{tabular}
\end{tabular}
\end{table}

Now we consider the case $p=2$. Given any characteristic subgroup $H$ of $G$, as in the proof of Theorem \ref{charreg} we can define the $n$-tuple $\h m$ by $m_i=\max\{a_i : \h a \in \Lambda(G), T(\h a) \subseteq H\}$. We say then that $H$ is a characteristic subgroup \emph{below $\h m$}. For a canonical tuple $\h m$, an example of a characteristic subgroup below $\h m$ is $R(\h m)$; when $p\neq 2$, this is the unique such subgroup, as Theorem \ref{charreg} shows. When $p=2$, there may be several characteristic subgroups below a given canonical tuple $\h m$. The set of such subgroups will be denoted $\Char_{\h m}(G)$. Our goal now is to give a description of these subgroups.

\begin{defn}
A canonical tuple $\h a \in \m C(G)$ is \emph{degenerate} at coordinate $i$ if
\begin{enumerate}
\item[(I)] $a_i=a_{i-1}$, or
\item[(II)] $a_{i+1}-a_i=\lambda_{i+1}-\lambda_i$,
\end{enumerate}
i.e., one of the bounds in Definition \ref{defcan} is sharp.
\end{defn}

We observe that, given a canonical tuple $\h a$, if $a_i=0$ then condition (I) of Definition \ref{defcan} implies $a_{i-1}=0$, so that $\h a$ is type (I) degenerate at coordinate $i$. The following Lemma, on the other hand, gives a simple but useful characterization of when $\h a$ is degenerate at $i$, provided $a_i \neq 0$:

\begin{lemma}\label{degen}
Let $\h a \in \m C(G)$ be a canonical tuple, and let $i\in\{1,\dots,n\}$ be given with $a_i \neq 0$. Then $\h a$ is degenerate at $i$ if and only if $\h a-\h e_i$ is noncanonical (at $i$). Moreover, if $\h a$ is degenerate at $i$ then $O(\h a-\h e_i)=O(\h a)$.
\end{lemma}
\begin{proof}
The first claim follows directly from the definition of degenerate. The last claim follows from the first by Lemma \ref{ncanup3}.
\end{proof}

\begin{defn}
A subgroup $H$ of a direct product $K_1 \times K_2 \times \cdots \times K_l$ is \emph{projection-surjective} if $\pi_i(H)=K_i$ for each $i\in\{1,\dots,l\}$, where $\pi_i$ is the natural projection map onto the $i$th component of the product. (In other words, $H$ is a subdirect product of $K_1, K_2, \dots, K_l$.)
\end{defn}

\begin{thm}\label{charirreg}
Let $G$ be an abelian 2-group. Given $\h m \in \m C(G)$, the characteristic subgroups below $\h m$ are in one-to-one correspondence with the projection-surjective subgroups of $Z_2^r$, where $r$ is the number of nondegenerate coordinates of $\h m$.
\end{thm}
\begin{proof}
Let $H$ be any characteristic subgroup below $\h m$. By the definition of $\h m$, for each $i$ there is some type $T(\h a) \subseteq H$ with $a_i=m_i$. Then $g=\prod_{j=1}^n t_{j,a_j}$ and $g'=t_{i,m_i}\prod_{j\neq i} t_{j,a_j}^{-1}$ are two elements of $T(\h a)$. Since $H$ is a subgroup, $gg' = t_{i,m_i}^2 = t_{i,\ul{m_i-1}}\in H$, where $\ul{x}$ is the ``clipping" function defined by
$$\ul x=\begin{cases}x, & \text{if $x\geq 0$}, \\ 0, & \text{if $x<0$}.\end{cases}$$
If we define $\h m'$ by $m_i' = \ul{m_i-1}$, then it is clear that $R(\h m') \subseteq H$, since the set $\{t_{i,\ul{m_i-1}} : i =1,\dots,n\}$ generates $R(\h m')$. Clearly $R(\h m') \subseteq R(\h m)$ and $R(\h m)/R(\h m') \cong Z_2^l$ where $l$ is the number of nonzero entries of $\h m$. To be more specific, let $\pi : R(\h m) \to R(\h m)/R(\h m')$ be the natural projection map, and set $K_i=\pi(\langle t_{i,m_i} \rangle)$; then $K_i \cong Z_2$ if $m_i\neq 0$, while $K_i$ is trivial if $m_i=0$. Let $k_i$ be the generator for $K_i$ (so $|k_i|=2$ unless $m_i=0$, in which case $k_i=1$). The lattice isomorphism theorem implies that the subgroups of $R(\h m)$ containing $R(\h m')$ (among which are all the subgroups $H$ in $\Char_{\h m}(G)$) are in one-to-one correspondence with subgroups of $\pi(R(\h m)) \cong Z_2^l$. Now, note that by definition, for any $i$, if $m_i=0$ then $i$ is a degenerate coordinate. If $i$ is a nonzero degenerate coordinate of $\h m$, define $\h a'$ by $a_j'=a_j$ for all $j\neq i$ and $a_i'=a_i-1$. Then observe that the degeneracy of $i$ ensures $O(\h a')=O(\h a)$ by Lemma \ref{degen}. Thus $T(\h a') \subseteq H$, and so $\hat g=\prod_{i=1}^n t_{1,a_i'} \in H$. If we write $\pi(g)=\prod_{j=1}^n k_j^{\epsilon_j}$, where each $\epsilon_j \in \{0,1\}$, then $\pi(\hat g)=\prod_{j\neq i} k_j^{\epsilon_j}$. Since $a_i=m_i$, we have $\epsilon_i=1$, and it follows that $\pi(g\hat g)=k_i^{\epsilon_i}\prod_{j\neq i} k_j^{2\epsilon_j} = k_i$, so that $K_i \leq \pi(H)$. Thus, if $D$ is the set of nonzero degenerate coordinates of $\h m$, we may write
$$ \pi(H) = K \times \prod_{j\in D} K_j,$$
where $K$ is a projection-surjective subgroup of $\prod_{j\in D'} K_j$, where $D'$ is the set of nondegenerate coordinates of $\h m$. This gives us an injective map $H \mapsto K$ from $\Char_{\h m}(G)$ into the set of projection-surjective subgroups of $\prod_{j \in D'} K_j \cong Z_2^r$. It remains only to show that this correspondence is surjective.

So let $K$ be an arbitrary projection-surjective subgroup of $\prod_{j \in D'} K_j$. Set $K'=K \times \prod_{j\in D} K_j$ and let $H=\pi^{-1}(K')$. The projection-surjectivity of $K$ ensures that $H$ is a subgroup below $\h m$. We only need to show that $H$ is characteristic. To do this, it is enough to show that if $T(\h a)$ is a noncanonical type contained in $H$ then there is another type $T(\h a')$ contained in $H$ with $\h a' > \h a$. Since $H$ contains the characteristic subgroup $R(\h m')$, it is sufficient to consider the case where $T(\h a)$ is not contained in $R(\h m')$, namely $\h a > \h m'$. Since $T(\h a)$ is noncanonical, there is some $i$ such that either $a_{i-1} > a_i$ or $a_{i+1}-a_i > \lambda_{i+1}-\lambda_i$. In the former case, we have $a_i < a_{i-1} \leq m_{i-1} \leq m_i$ since by condition (I) of $\h m$ being canonical, while in the latter case, we have
$a_i < a_{i+1} - (\lambda_{i+1}-\lambda_i) \leq a_{i+1} - (m_{i+1}-m_i) \leq a_{i+1} - (a_{i+1}-m_i) = m_i$. So in either case we have $a_i < m_i$, which implies $a_i=m_i-1$, since $\h a \geq \h m'$. Now if every such coordinate $i$ was nondegenerate in $\h m$, then by repeated application of Lemma \ref{degen}, $\h a$ would be canonical, contrary to assumption. So there must be some such $i$ which is a degenerate coordinate of $\h m$. Define $\h a'$ by $a_j'=a_j$ for $j\neq i$ and $a_i'=m_i$. Let $g$ be an element of type $T(\h a)$ and write $k=\pi(g)=\prod_{j=1}^n k_j^{\epsilon_j}$ with $\epsilon_j \in \{0,1\}$ (namely, we will have $\epsilon_j=1$ if and only if $a_j=m_j$). Then $k'=k_i \prod_{j\neq i} k_j^{\epsilon_j}$ is also in $\pi(H)$ (since $k_i \in \prod_{j\in D} K_j \subseteq K'$), and the set $\pi^{-1}(k')$ includes elements of type $\h a'$, so $T(\h a') \subseteq H$, as desired.
\end{proof}

A statement equivalent to the following is stated (without proof) in \cite[p. 23]{miller}:
\begin{cor}\label{irreg_exist}
Given an abelian 2-group $G$, an irregular characteristic subgroup below a canonical tuple $\h a\in\m C(G)$ exists if and only if $\h a$ has at least two nondegenerate coordinates.
\end{cor}
\begin{proof}
Since $Z_2^k$ has proper projection-surjective subgroups if and only if $k\geq 2$, this follows from Theorem \ref{charirreg}.
\end{proof}

\begin{example}
Let $G=Z_2\times Z_8$. Let $H_1,\dots,H_6$ be the regular characteristic subgroups of $G$, as shown in Table \ref{tabz2z8}. We note that (1,2) is the only canonical tuple with two nondegenerate coordinates; consequently, there is an irregular characteristic subgroup $K$ below (1,2) and this is the only irregular characteristic subgroup of $G$. Note that the lattice of characteristic subgroups of $G$ is not distributive, in contrast to Theorem \ref{cordist}; see Theorem \ref{nondist} below.
\end{example}

\begin{table}
\centering
\caption{Characteristic subgroups of $Z_2\times Z_8$}\label{tabz2z8}
\begin{tabular}{lr}
\begin{tabular}{|l|l|} \hline
$H_1$&R(0,0)\\ \hline
$H_2$&R(0,1)\\ \hline
$H_3$&R(0,2)\\ \hline
$H_4$&R(1,1)\\ \hline
$K$&$\text{R}(0,1)\cup\text{T}(1,2)$\\ \hline
$H_5$&R(1,2)\\ \hline
$H_6$&R(1,3)\\ \hline
\end{tabular}
&
\begin{tabular}{r}
\psset{xunit=.2cm,yunit=.2cm,labelsep=2.5mm}
\begin{pspicture*}(-10,-5)(10,23)
\Rput[t](0,0){$H_1$}

\Rput[t](0,5){$H_2$}
\psline(0,-.5)(0,1.5)

\Rput[t](-5,10){$H_3$}
\Rput[t](5,10){$H_4$}
\Rput[t](0,10){$K$}
\psline(-1.5,4.5)(-3.5,6.5)
\psline(1.5,4.5)(3.5,6.5)
\psline(0,4.5)(0,6.5)

\Rput[t](0,15){$H_5$}
\psline(-3.5,9.5)(-1.5,11.5)
\psline(3.5,9.5)(1.5,11.5)
\psline(0,9.5)(0,11.5)

\Rput[t](0,20){$H_6$}
\psline(0,14.5)(0,16.5)
\end{pspicture*}
\end{tabular}
\end{tabular}
\end{table}

The following theorem will not be needed in what follows; however, it is of interest because it, together with Theorem \ref{charirreg}, enables one to enumerate the characteristic subgroups of any finite abelian 2-group, and hence of any finite abelian group (as an example, see Table \ref{numchars}):

\begin{thm}\label{numproj}
The number of projection-surjective subgroups of $Z_2^k$ is
$$n_k=\sum_{i=0}^k(-1)^{i+k}\binom{k}{i}\sum_{j=0}^i{\binom{i}{j}}_2,$$
where ${\binom{i}{j}}_2$ are the Gaussian binomial coefficients given by
$${\binom{i}{j}}_2=\frac{\displaystyle\prod_{l=0}^{j-1}\left(2^{i-l}-1\right)}
{\displaystyle\prod_{l=1}^j\left(2^l-1\right)}.$$
\end{thm}
\begin{remark}
The sequence $n_k$ begins $1,1,2,6,26,158,1330,15414,245578,5382862,\dots$ for $k=0,1,2,\dots$ and may be found as A135922 of Sloane's on-line encyclopedia of integer sequences \cite{sloane}.
\end{remark}
\begin{proof}
Let $X=\{1,\dots,k\}$. For any subgroup $H$ of $Z_2^k$, set $\rho(H)$ denote the set of integers $i\in X$ such that $\pi_i(H)=Z_2$. So $H$ is projection-surjective if and only if $\rho(H)=X$. For any subset $Y \subseteq X$, let $n(Y)$ be the number of subgroups $H$ of $Z_2^k$ such that $\rho(H)=Y$. We would like to compute $n_k=n(X)$. Now define $m(Y)$ to be the number of subgroups $H$ of $Z_2^k$ with $\rho(H) \subseteq Y$. So
$$m(Y) =\sum_{Z \subseteq Y} n(Z).$$
Now $m(Y)$ is simply the total number of subgroups of $Z_2^{|Y|}$; this is the same as the number of subspaces of a $|Y|$-dimensional vector space over $\F_2$. Since the number of $j$-dimensional subspaces of such a vector space is known to be
$$\frac{\displaystyle\prod_{l=0}^{j-1} \left(2^{|Y|}-2^l\right)}{\displaystyle\prod_{l=0}^{j-1} \left(2^j-2^l\right)}
= \frac{\displaystyle\prod_{l=0}^{j-1} \left(2^{|Y|-l}-1\right)}{\displaystyle\prod_{l=0}^{j-1} \left(2^{j-l}-1\right)}
= \frac{\displaystyle\prod_{l=0}^{j-1} \left(2^{|Y|-l}-1\right)}{\displaystyle\prod_{l=1}^{j} \left(2^l-1\right)}
={\binom{|Y|}{j}}_2
$$
(see, e.g., \cite[p. 412]{dummit}), it follows that
$$m(Y) =\sum_{j=0}^{|Y|}{\binom{|Y|}{j}}_2.$$
We note, in particular, that $m(Y)$ only depends on the size of $Y$. By the inclusion-exclusion principle (see, e.g., \cite[p. 185]{brualdi}) we have
\begin{align*}
n(X)&=\sum_{Y \subseteq X} (-1)^{|Y|+|X|}m(Y)\\
&=\sum_{i=0}^k \sum_{\overset{Y \subseteq X}{|Y|=i}} (-1)^{i+k}m(Y)\\
&=\sum_{i=0}^k \binom{k}{i}(-1)^{i+k}m(\{1,\dots,i-1\})\\
&=\sum_{i=0}^k (-1)^{i+k}\binom{k}{i}\sum_{j=0}^{i}{\binom{i}{j}}_2,
\end{align*}
as desired.
\end{proof}

\begin{table}
\centering
\caption{Number of characteristic subgroups of $Z_2\times Z_{2^2}\times Z_{2^3}\times \cdots\times Z_{2^n}$}\label{numchars}
\begin{tabular}{|l|r|} \hline
1 & 2 \\ \hline
2 & 4 \\ \hline
3 & 9 \\ \hline
4 & 21 \\ \hline
5 & 52 \\ \hline
6 & 134 \\ \hline
7 & 363 \\ \hline
8 & 1027 \\ \hline
9 & 3054 \\ \hline
10 & 9516 \\ \hline
11 & 31229 \\ \hline
12 & 107745 \\ \hline
13 & 392792 \\ \hline
14 & 1511010 \\ \hline
15 & 6167551 \\ \hline
16 & 26670383 \\ \hline
17 & 122982386 \\ \hline
18 & 603221064 \\ \hline
19 & 3172965937 \\ \hline
20 & 17817816493 \\ \hline
21 & 107984192188 \\ \hline
22 & 700497542494 \\ \hline
23 & 4939837336979 \\ \hline
24 & 37315530126171 \\ \hline
25 & 309078760337078 \\ \hline
26 & 2736173394567076 \\ \hline
27 & 26852600855758373 \\ \hline
28 & 279765993533235769 \\ \hline
29 & 3279737127172518880 \\ \hline
30 & 40284238921560357658 \\ \hline
31 & 568574087799302502375 \\ \hline
32 & 8225663800386744379975 \\ \hline
33 & 140886928953442040025658 \\ \hline
34 & 2392158426272284053385152 \\ \hline
35 & 50137841812585275382579929 \\ \hline
36 & 993099669210856047011613573 \\ \hline
37 & 25701228868609248542152214980 \\ \hline
38 & 589013066872810742690824633750 \\ \hline
39 & 19005348215516204077748683286267 \\ \hline
40 & 498993627095578092364760281155059 \\ \hline
\end{tabular}

\end{table}

We will have need of the following theorem later on:

\begin{thm}\label{osplit}
Let $G$ be an abelian $p$-group (for any prime $p$) with no repeated factors (i.e., $0 < \lambda_1 < \lambda_2 < \cdots < \lambda_n$), and let $\h a\in \m C(G)$ be a canonical tuple. Then
$$O(\h a)=\bigcup\{T(\h b) : \h b \leq \h a \text{ and, for each nondegenerate coordinate $i$ of $\h a$, } b_i=a_i\}$$
\end{thm}
\begin{proof}
Let
$$A=\bigcup\{T(\h b) : \h b \leq \h a \text{ and, for each nondegenerate coordinate $i$ of $\h a$, } b_i=a_i\}.$$
We first show $O(\h a) \subseteq A$. Given any $T(\h b) \subseteq O(\h a)$, we have $\h b \leq \h a$ since $T(\h a)$ is the maximum type in $O(\h a)$ by Theorem \ref{unican}. Now let $i$ be a nondegenerate coordinate of $\h a$ and suppose $b_i<a_i$. Then $\h a'=\h a-\h e_i$ is canonical by Lemma \ref{degen}, hence by Theorem \ref{canchar} $R(\h a')$ is a characteristic subgroup with $\h b \leq \h a'$, so $O(\h a)=O(\h b)\subseteq R(\h a')$, which is a contradiction since $\h a \nleq \h a'$. Consequently $b_i=a_i$, which proves $O(\h a) \subseteq A$.

Now we must show $A \subseteq O(\h a)$. Suppose there is some $T(\h b) \subseteq A$ with $T(\h b) \nsubseteq O(\h a)$, i.e. $O(\h b)\neq O(\h a)$. Take a maximal such $\h b$. We must then have $\h b < \h a$. Let $i$ be the first coordinate for which $b_i < a_i$. Then, by the definition of $A$, $i$ must be a degenerate coordinate of $\h a$. If $i$ is type (I) degenerate, then $a_i=a_{i-1}$, so we have $b_i<a_i=a_{i-1}=b_{i-1}$, and by Lemma \ref{ncanup3}, $O(\h b+\h e_i)=O(\h b)\neq O(\h a)$, while $\h b+\h e_i>\h b$, contradicting the maximality of $\h b$. On the other hand, if $i$ is type (II) degenerate, then $a_i+\lambda_{i+1}-\lambda_i=a_{i+1}$, then let $j$ be the first coordinate greater than $i$ such that $b_j=a_j$; such a $j$ must exist since otherwise all the coordinates $i,\dots,n$ of $\h a$ would be degenerate and we would have $a_n=a_{n-1}=\cdots=a_{i+1}=a_i$, contradicting $a_i+\lambda_{i+1}-\lambda_i=a_{i+1}$ since $\lambda_{i+1} \neq \lambda_i$. Thus all of the coordinates $i,\dots,j-1$ of $\h a$ are degenerate. We find that each coordinate $k\in\{i,\dots,j-1\}$ is type (II) degenerate, i.e. we find that $a_k+\lambda_{k+1}-\lambda_k=a_{k+1}$: For $k=i$ this holds by assumption, while for $k>i$, if $k$ were degenerate of the first type, i.e. $a_k=a_{k-1}$, we would have a contradiction since  by induction, $a_{k-1}+\lambda_k-\lambda_{k-1}=a_k$ and $\lambda_k\neq\lambda_{k-1}$. So we have $b_{j-1}+\lambda_j-\lambda_{j-1} < a_{j-1}+\lambda_j-\lambda_{j-1}=a_j=b_j$, so by Lemma \ref{ncanup3}, we again obtain a $\h b+\h e_i>\h b$ with $O(\h b+\h e_i)=O(\h b)\neq O(\h a)$, contradicting the maximality of $\h b$.
\end{proof}

\begin{example}
Let $G=Z_p\times Z_{p^3}\times Z_{p^5}$. The first and third coordinates of the tuple $(1,3,3)$ are degenerate. So we have
\begin{align*}
O(1,3,3)&= T(0,3,0)\cup T(0,3,1)\cup T(0,3,2) \cup T(0,3,3) \\
& \cup T(1,3,0) \cup T(1,3,1) \cup T(1,3,2) \cup T(1,3,3).
\end{align*}
\end{example}

\section{Isomorphic Lattices of Characteristic Subgroups}\label{charlatiso}

We now turn to our main question: When do two finite abelian groups have isomorphic lattices of characteristic subgroups? The following theorems give some fundamental examples of when this can occur:

\begin{thm}\label{corpq}
Let $G=Z_{p^{\lambda_1}}\times Z_{p^{\lambda_2}}\times \cdots \times Z_{p^{\lambda_n}}$ be an abelian $p$-group with $p \neq 2$. Let $q\neq 2$ be any other odd prime, and set $G'=Z_{q^{\lambda_1}}\times Z_{q^{\lambda_2}}\times \cdots \times Z_{q^{\lambda_n}}$. Then $\Char(G) \cong \Char(G')$.
\end{thm}
\begin{proof}
This is immediate from Corollary \ref{cordist} since $\Char(G) \cong \m C(G) \cong \m C(G') \cong \Char(G')$.
\end{proof}

The next theorem shows that, in the case $p\neq 2$, adding a duplicate factor in the direct decomposition of $G$ does not change its lattice of characteristic subgroups.
\begin{thm}\label{duplat}
Let $G=Z_{p^{\lambda_1}}\times Z_{p^{\lambda_2}}\times \cdots \times Z_{p^{\lambda_n}}$ be an abelian $p$-group with $p \neq 2$. Then for any $i\in\{1,\dots,n\}$, $\Char(G)\cong \Char(G \times Z_{p^{\lambda_i}})$.
\end{thm}
\begin{proof}
Let $G'=Z_{p^{\lambda_1}}\times Z_{p^{\lambda_2}}\times \cdots \times Z_{p^{\lambda_{i-1}}} \times Z_{p^{\lambda_i}} \times Z_{p^{\lambda_i}} \times Z_{p^{\lambda_{i+1}}} \times \cdots \times Z_{p^{\lambda_n}}$, so $G' \cong G \times Z_{p^{\lambda_i}}$. Every canonical tuple of $G'$ has the form $(a_1,a_2,\dots,a_{i-1},a_i,a_i,a_{i+1},\dots,a_n)$, i.e., the $i$th and $(i+1)$th coordinates are forced to be equal. It follows that the correspondence
$$R(a_1,\dots,a_n) \mapsto R(a_1,\dots,a_{i-1},a_i,a_i,a_{i+1},\dots,a_n)$$
is an isomorphism of $\Char(G)$ onto $\Char(G')$.
\end{proof}

\begin{thm}\label{thmchain}
The lattice of characteristic subgroups of a finite abelian group $G$ is a chain if and only if $G \cong Z_{p^k}^{\mu_1} \times Z_{p^{k+1}}^{\mu_2}$ for some natural numbers $k, \mu_1,\mu_2 \geq 0$ and some prime $p$.
\end{thm}
\begin{proof}
First assume the lattice of characteristic subgroups of $G$ is a
chain. If $|G|$ were not a prime power, it would have distinct prime
divisors $p$ and $q$, and the Sylow $p$-subgroup and Sylow
$q$-subgroup of $G$ would be incomparable. So $G$ must be an abelian
$p$-group, and without loss of generality we may write
$G=Z_{p^{\lambda_1}}\times \cdots \times Z_{p^{\lambda_n}}$, where
$1 \leq \lambda_1 \leq \lambda_2 \leq \dots \leq \lambda_n$. Note
that the claim that $G$ has the form $Z_{p^k}^{\mu_1} \times
Z_{p^{k+1}}^{\mu_2}$ is equivalent to the claim that
$\lambda_n-\lambda_1 \leq 1$. So suppose $\lambda_n-\lambda_1 \geq
2$. Define tuples $\h a$ and $\h a'$ by
\begin{align*}
a_i&=1\\
a_i'&=\lambda_i-\lambda_1.
\end{align*}
for all $i=1,\dots,n$. Then it is easy to see that $\h a$ and $\h
a'$ are canonical tuples. Since $a_1=1 > 0 = a_1'$ we have $\h a
\nless \h a'$, while since $a_n=1 < 2 \leq \lambda_n-\lambda_1 =
a_n'$, we have $\h a \ngtr \h a'$. The characteristic subgroups
$R(\h a)$ and $R(\h a')$ are then incomparable, contradicting the
hypothesis. Hence $\lambda_n-\lambda_1 \leq 1$, as desired.

Conversely, suppose $\lambda_n-\lambda_1 \leq 1$. Then every
canonical tuple $\h a\in \m C(G)$ has the form
$$a_i = \begin{cases}0, &\text{ if $i < j$} \\ 1, &\text{ if $i \geq j$} \end{cases}$$
for some natural number $j\geq 0$. In the case $p=2$, since such a
tuple has at most one nondegenerate coordinate, it follows from
Theorem \ref{charirreg} that every characteristic subgroup of $G$ is
regular. (Since $Z_2^k$ has only one projection-surjective subgroup
if $k\in\{0,1\}$, there is a unique characteristic subgroup below
each canonical tuple $\h a$, namely $R(\h a)$.) Since any two such
tuples $\h a$ and $\h a'$ are clearly comparable, it follows that
$R(\h a)$ and $R(\h a')$ are comparable, so $\Char(G)$ is a chain.
\end{proof}

\begin{thm}\label{z25z124}
For any prime $p$, $\Char(Z_{p^2} \times Z_{p^5}) \cong \Char(Z_p \times Z_{p^2}\times Z_{p^4})$.
\end{thm}
\begin{proof}
This is clear upon examination of Tables \ref{tab25} and \ref{tab25p2}.
\end{proof}

\begin{table}
\centering
\caption{Characteristic subgroups of $\Char(Z_{p^2} \times Z_{p^5})$ and $\Char(Z_p \times Z_{p^2}\times Z_{p^4})$, $p\neq 2$}\label{tab25}
\begin{tabular}{ll|ll}
\begin{tabular}{|l|l|} \hline
$H_1$&R(0,0)\\ \hline
$H_2$&R(0,1)\\ \hline
$H_3$&R(0,2)\\ \hline
$H_4$&R(0,3)\\ \hline
$H_5$&R(1,1)\\ \hline
$H_6$&R(1,2)\\ \hline
$H_7$&R(1,3)\\ \hline
$H_8$&R(1,4)\\ \hline
$H_9$&R(2,2)\\ \hline
$H_{10}$&R(2,3)\\ \hline
$H_{11}$&R(2,4)\\ \hline
$H_{12}$&R(2,5)\\ \hline
\end{tabular}
&
\begin{tabular}{r}
\psset{xunit=.2cm,yunit=.2cm,labelsep=2.5mm}
\begin{pspicture*}(-10,-5)(10,37)
\Rput[t](0,0){$H_1$}

\Rput[t](0,5){$H_2$}
\psline(0,-.5)(0,1.5)

\Rput[t](-5,10){$H_3$}
\Rput[t](5,10){$H_5$}
\psline(-1.5,4.5)(-3.5,6.5)
\psline(1.5,4.5)(3.5,6.5)

\Rput[t](-5,15){$H_4$}
\Rput[t](5,15){$H_6$}
\psline(-5,9.5)(-5,11.5)
\psline(5,9.5)(5,11.5)
\psline(-3.5,9.5)(3.5,11.5)

\Rput[t](-5,20){$H_7$}
\Rput[t](5,20){$H_9$}
\psline(-5,14.5)(-5,16.5)
\psline(5,14.5)(5,16.5)
\psline(3.5,14.5)(-3.5,16.5)

\Rput[t](-5,25){$H_8$}
\Rput[t](5,25){$H_{10}$}
\psline(-5,19.5)(-5,21.5)
\psline(5,19.5)(5,21.5)
\psline(-3.5,19.5)(3.5,21.5)

\Rput[t](0,30){$H_{11}$}
\psline(-3.5,24.5)(-1.5,26.5)
\psline(3.5,24.5)(1.5,26.5)

\Rput[t](0,35){$H_{12}$}
\psline(0,29.5)(0,31.5)
\end{pspicture*}
\end{tabular}

&
\begin{tabular}{|l|l|} \hline
$H_1'$&R(0,0,0)\\ \hline
$H_2'$&R(0,0,1)\\ \hline
$H_3'$&R(0,0,2)\\ \hline
$H_4'$&R(0,1,1)\\ \hline
$H_5'$&R(0,1,2)\\ \hline
$H_6'$&R(0,1,3)\\ \hline
$H_7'$&R(1,1,1)\\ \hline
$H_8'$&R(1,1,2)\\ \hline
$H_9'$&R(1,1,3)\\ \hline
$H_{10}'$&R(1,2,2)\\ \hline
$H_{11}'$&R(1,2,3)\\ \hline
$H_{12}'$&R(1,2,4)\\ \hline
\end{tabular}
&
\begin{tabular}{r}
\psset{xunit=.2cm,yunit=.2cm,labelsep=2.5mm}
\begin{pspicture*}(-10,-5)(10,37)
\Rput[t](0,0){$H_1'$}

\Rput[t](0,5){$H_2'$}
\psline(0,-.5)(0,1.5)

\Rput[t](-5,10){$H_4'$}
\Rput[t](5,10){$H_3'$}
\psline(-1.5,4.5)(-3.5,6.5)
\psline(1.5,4.5)(3.5,6.5)

\Rput[t](-5,15){$H_7'$}
\Rput[t](5,15){$H_5'$}
\psline(-5,9.5)(-5,11.5)
\psline(5,9.5)(5,11.5)
\psline(-3.5,9.5)(3.5,11.5)

\Rput[t](-5,20){$H_8'$}
\Rput[t](5,20){$H_6'$}
\psline(-5,14.5)(-5,16.5)
\psline(5,14.5)(5,16.5)
\psline(3.5,14.5)(-3.5,16.5)

\Rput[t](-5,25){$H_{10}'$}
\Rput[t](5,25){$H_9'$}
\psline(-5,19.5)(-5,21.5)
\psline(5,19.5)(5,21.5)
\psline(-3.5,19.5)(3.5,21.5)

\Rput[t](0,30){$H_{11}'$}
\psline(-3.5,24.5)(-1.5,26.5)
\psline(3.5,24.5)(1.5,26.5)

\Rput[t](0,35){$H_{12}'$}
\psline(0,29.5)(0,31.5)
\end{pspicture*}
\end{tabular}
\end{tabular}
\end{table}

\begin{table}
\centering
\caption{Characteristic subgroups of $\Char(Z_{p^2} \times Z_{p^5})$ and $\Char(Z_p \times Z_{p^2}\times Z_{p^4})$, $p=2$}\label{tab25p2}
\begin{tabular}{ll|ll}
\begin{tabular}{|l|l|} \hline
$H_1$&R(0,0)\\ \hline
$H_2$&R(0,1)\\ \hline
$H_3$&R(0,2)\\ \hline
$H_4$&R(0,3)\\ \hline
$H_5$&R(1,1)\\ \hline
$H_6$&R(1,2)\\ \hline
$H_7$&R(1,3)\\ \hline
$H_8$&R(1,4)\\ \hline
$H_9$&R(2,2)\\ \hline
$H_{10}$&R(2,3)\\ \hline
$H_{11}$&R(2,4)\\ \hline
$H_{12}$&R(2,5)\\ \hline
$K_1$&$H_2\cup \text{T}$(1,2) \\ \hline
$K_2$&$H_3\cup \text{T}$(1,3) \\ \hline
$K_3$&$H_6\cup \text{T}$(2,3) \\ \hline
$K_4$&$H_7\cup \text{T}$(2,4) \\ \hline
\end{tabular}
&
\begin{tabular}{r}
\psset{xunit=.2cm,yunit=.2cm,labelsep=2.5mm}
\begin{pspicture*}(-7.5,-5)(7.5,37)
\Rput[t](0,0){$H_1$}

\Rput[t](0,5){$H_2$}
\psline(0,-.5)(0,1.5)

\Rput[t](-5,10){$H_3$}
\Rput[t](5,10){$H_5$}
\Rput[t](0,10){$K_1$}
\psline(-1.5,4.5)(-3.5,6.5)
\psline(1.5,4.5)(3.5,6.5)
\psline(0,4.5)(0,6.5)

\Rput[t](-5,15){$H_4$}
\Rput[t](5,15){$H_6$}
\Rput[t](0,15){$K_2$}
\psline(-5,9.5)(-5,11.5)
\psline(5,9.5)(5,11.5)
\psline(-3.5,9.5)(3.5,11.5)
\psline(1.5,9.5)(4.25,11.5)
\psline(-4.25,9.5)(-1.5,11.5)

\Rput[t](-5,20){$H_7$}
\Rput[t](5,20){$H_9$}
\Rput[t](0,20){$K_3$}
\psline(-5,14.5)(-5,16.5)
\psline(5,14.5)(5,16.5)
\psline(3.5,14.5)(-3.5,16.5)
\psline(-1.5,14.5)(-4.25,16.5)
\psline(4.25,14.5)(1.5,16.5)

\Rput[t](-5,25){$H_8$}
\Rput[t](5,25){$H_{10}$}
\Rput[t](0,25){$K_4$}
\psline(-5,19.5)(-5,21.5)
\psline(5,19.5)(5,21.5)
\psline(-3.5,19.5)(3.5,21.5)
\psline(1.5,19.5)(4.25,21.5)
\psline(-4.25,19.5)(-1.5,21.5)

\Rput[t](0,30){$H_{11}$}
\psline(-3.5,24.5)(-1.5,26.5)
\psline(3.5,24.5)(1.5,26.5)
\psline(0,24.5)(0,26.5)

\Rput[t](0,35){$H_{12}$}
\psline(0,29.5)(0,31.5)
\end{pspicture*}
\end{tabular}

&
\begin{tabular}{|l|l|} \hline
$H_1'$&R(0,0,0)\\ \hline
$H_2'$&R(0,0,1)\\ \hline
$H_3'$&R(0,0,2)\\ \hline
$H_4'$&R(0,1,1)\\ \hline
$H_5'$&R(0,1,2)\\ \hline
$H_6'$&R(0,1,3)\\ \hline
$H_7'$&R(1,1,1)\\ \hline
$H_8'$&R(1,1,2)\\ \hline
$H_9'$&R(1,1,3)\\ \hline
$H_{10}'$&R(1,2,2)\\ \hline
$H_{11}'$&R(1,2,3)\\ \hline
$H_{12}'$&R(1,2,4)\\ \hline
$K_1'$&$H_2'\cup \text{T}$(0,1,2)\\ \hline
$K_2'$&$H_4'\cup \text{T}$(1,1,2)\\ \hline
$K_3'$&$H_5'\cup \text{T}$(1,1,3)\\ \hline
$K_4'$&$H_8'\cup \text{T}$(1,2,3)\\ \hline
\end{tabular}
&
\begin{tabular}{r}
\psset{xunit=.2cm,yunit=.2cm,labelsep=2.5mm}
\begin{pspicture*}(-7.5,-5)(7.5,37)
\Rput[t](0,0){$H_1'$}

\Rput[t](0,5){$H_2'$}
\psline(0,-.5)(0,1.5)

\Rput[t](-5,10){$H_4'$}
\Rput[t](5,10){$H_3'$}
\Rput[t](0,10){$K_1$}
\psline(-1.5,4.5)(-3.5,6.5)
\psline(1.5,4.5)(3.5,6.5)
\psline(0,4.5)(0,6.5)

\Rput[t](-5,15){$H_7'$}
\Rput[t](5,15){$H_5'$}
\Rput[t](0,15){$K_2'$}
\psline(-5,9.5)(-5,11.5)
\psline(5,9.5)(5,11.5)
\psline(-3.5,9.5)(3.5,11.5)
\psline(1.5,9.5)(4.25,11.5)
\psline(-4.25,9.5)(-1.5,11.5)

\Rput[t](-5,20){$H_8'$}
\Rput[t](5,20){$H_6'$}
\Rput[t](0,20){$K_3'$}
\psline(-5,14.5)(-5,16.5)
\psline(5,14.5)(5,16.5)
\psline(3.5,14.5)(-3.5,16.5)
\psline(-1.5,14.5)(-4.25,16.5)
\psline(4.25,14.5)(1.5,16.5)

\Rput[t](-5,25){$H_{10}'$}
\Rput[t](5,25){$H_9'$}
\Rput[t](0,25){$K_4'$}
\psline(-5,19.5)(-5,21.5)
\psline(5,19.5)(5,21.5)
\psline(-3.5,19.5)(3.5,21.5)
\psline(1.5,19.5)(4.25,21.5)
\psline(-4.25,19.5)(-1.5,21.5)

\Rput[t](0,30){$H_{11}'$}
\psline(-3.5,24.5)(-1.5,26.5)
\psline(3.5,24.5)(1.5,26.5)
\psline(0,24.5)(0,26.5)

\Rput[t](0,35){$H_{12}'$}
\psline(0,29.5)(0,31.5)
\end{pspicture*}
\end{tabular}
\end{tabular}
\end{table}

\begin{thm}\label{nondist}
The lattice of characteristic subgroups of an abelian 2-group $G$ is distributive if and only if all of its characteristic subgroups are regular.
\end{thm}
\begin{proof}
The ``if" part is trivial, since the lattice of regular characteristic subgroups is distributive, being isomorphic to the lattice $\m C(G)$ of canonical tuples. So suppose there is an irregular characteristic subgroup $K$ below a tuple $\h m$. By Corollary \ref{irreg_exist}, there must be at least two distinct nondegenerate coordinates $i$ and $j$ of $\h m$. Define $\h a=\h m-\h e_i$, $\h a'=\h m-\h e_j$, and $\h b=\h m-\h e_i-\h e_j$, where the subtraction is defined component-wise. Then $\h m$, $\h a$, $\h a'$, and $\h b$ are all canonical. Then define $K'=R(\h b)\cup T(\h m)$, so $K'$ is another irregular characteristic subgroup below $\h m$. Since $R(\h a)$, $K'$, and $R(\h a')$ are distinct index 2 subgroups of $R(\h m)$ and each contains $R(\h b)$ as an index 2 subgroup, it follows that $R(\h b), R(\h a), K', R(\h a')$, and $R(\h m)$ form a diamond:

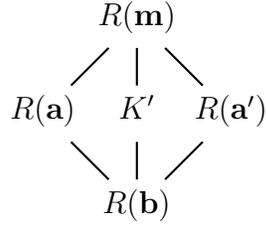
\begin{center}
\psset{xunit=.25cm,yunit=.25cm,labelsep=2.5mm}
\begin{pspicture*}(-15,-5)(15,12)
\Rput[t](0,0){$R(\h b)$}

\Rput[t](0,5){$K'$}
\Rput[t](-5,5){$R(\h a)$}
\Rput[t](5,5){$R(\h a')$}
\psline(0,-.5)(0,1.5)
\psline(-1.5,-.5)(-3.5,1.5)
\psline(1.5,-.5)(3.5,1.5)

\Rput[t](0,10){$R(\h m)$}
\psline(0,6.5)(0,4.5)
\psline(-1.5,6.5)(-3.5,4.5)
\psline(1.5,6.5)(3.5,4.5)
\end{pspicture*}
\end{center}

Thus, $\Char(G)$ is not distributive.
\end{proof}

The next theorem describes explicitly when the above situation does or does not occur:
\begin{thm}\label{noirreg}
Let $G=Z_{2^{\lambda_1}} \times \cdots \times Z_{2^{\lambda_n}}$ be an abelian 2-group, $\lambda_1 \leq \lambda_2 \leq \cdots \leq \lambda_n$. Then $G$ has an irregular characteristic subgroup if and only if there exist indices $i<j$ with $\lambda_j-\lambda_i \geq 2$ such that neither of the factors $Z_{2^{\lambda_i}}$ nor $Z_{2^{\lambda_j}}$ occur repeated in the decomposition of $G$.
\end{thm}
\begin{proof}
By Corollary \ref{irreg_exist}, this is equivalent to there being a canonical tuple $\h a\in\m C(G)$ with at least two nondegenerate coordinates. First assume there exist $i<j$ with $\lambda_j-\lambda_i \geq 2$ such that neither $\lambda_i$ nor $\lambda_j$ are repeated. Consider the tuple $\h a$ given by
$$a_k=\begin{cases}
0 & \text{for $k<i$} \\
1 & \text{for $i\leq k<j$} \\
2 & \text{for $k\geq j$}
\end{cases}$$
Since $\lambda_i$ and $\lambda_j$ are not repeated, we have $\lambda_{i-1}<\lambda_i$ and $\lambda_{j-1}<\lambda_j$, which ensures that $\h a$ is canonical. Likewise, we have $\lambda_i<\lambda_{i+1}$ and $\lambda_j<\lambda_{j+1}$ (provided $j+1\leq n$), which ensures that $\h a$ is nondegenerate at coordinates $i$ and $j$, as desired.

Now assume, conversely, that there is a canonical tuple $\h a\in\m C(G)$ with at least two nondegenerate coordinates $i$ and $j$. Without loss of generality, $i<j$. Since $\h a$ is nondegenerate at $i$ and $j$, $\lambda_i$ and $\lambda_j$ must not be repeated. Since $\h a$ is nondegenerate at $i$, we have
$$\lambda_{i+1}-\lambda_i > a_{i+1}-a_i \geq 0,$$
while since $\h a$ is nondegenerate at $j$, we have
$$\lambda_j-\lambda_{j-1} \geq a_j-a_{j-1} > 0.$$
If $j-i>2$, then this yields
$$\lambda_j-\lambda_i \geq (\lambda_j-\lambda_{j-1})+(\lambda_{i+1}-\lambda_i)\geq 2.$$
The only case remaining is $j-i=1$. But in this case also we must have $\lambda_j-\lambda_i\geq 2$, since otherwise the only remaining option would be to have $\lambda_{i+1}-\lambda_i=1$, hence either $a_{i+1}=a_i$ or $a_{i+1}=a_i+1$, which would imply $\h a$ is degenerate at $i+1$ or $i$, respectively.
\end{proof}

\begin{thm}\label{unimin}
Let $G$ be a nontrivial abelian $p$-group. Then $G$ has a unique minimal nontrivial characteristic subgroup.
\end{thm}
\begin{proof}
Write
$$G=Z_{2^{\lambda_1}}^{\alpha_1} \times Z_{2^{\lambda_2}}^{\alpha_2}\times \cdots\times Z_{2^{\lambda_n}}^{\alpha_n}$$
where
$$0 < \lambda_1 < \lambda_2 < \cdots < \lambda_n,$$
$$\alpha_1,\dots,\alpha_n \geq 1.$$
Set
$$s=\alpha_1+\alpha_2+\dots+\alpha_n.$$
Define a canonical tuple $\h r=(0,\dots,0,1,\dots,1)$, where the number of 1's is $\alpha_n$.
We claim that $R(\h r)$ is the minimum nontrivial characteristic subgroup of $G$. Let $H$ be any nontrivial characteristic subgroup of $G$. We will show $R(\h r) \subseteq H$. First suppose $p\neq 2$. Then $H=R(\h a)$ for some canonical tuple $\h a$ (by Theorem \ref{charreg}), and we must have the last coordinate $a_s \geq 1$ since otherwise condition (I) of Definition \ref{defcan} would imply $\h a=0$, i.e. $H=1$, a contradiction. Condition (II) of Definition \ref{defcan} now implies $a_i=a_s\geq 1$ for $i\in\{s-\alpha_n+1,\dots,s\}$. Thus $\h r \leq \h a$, hence $R(\h r) \subseteq H$.

So we may assume $p=2$. Now $H$ is a characteristic subgroup below some canonical tuple $\h m$. As in the proof of Theorem \ref{charirreg}, we know $R(\h m') \subseteq H$, where $\h m'$ is defined by $m_i'=\ul{m_i-1}$. If $m_s \geq 2$, then $m_s' \geq 1$, and $R(\h r) \subseteq R(\h m') \subseteq H$ as above. So we must have $m_s=1$. Then $\h m$ is of the form $\h m=(0,\dots,0,1,\dots,1)$. Since such a tuple has at most 1 nondegenerate coordinate, there are no irregular characteristic subgroups below $\h m$ (by Theorem \ref{irreg_exist}), so $H$ is regular and $H=R(\h m)$. Then, as above, $\h r \leq \h m$, hence $R(\h r) \subseteq H$, as desired.
\end{proof}

\begin{thm}\label{indecomp}
Let $G$ be an abelian $p$-group. Then $\Char(G)$ is directly indecomposable, i.e. it is not isomorphic to a direct product of two nontrivial lattices.
\end{thm}
\begin{proof}
Theorem \ref{unimin} says that $G$ has a unique minimal nontrivial characteristic subgroup, i.e. $\Char(G)$ has a unique atom. Since a decomposable lattice must have at least two atoms, the result follows.
\end{proof}

\section{Main result}

If $G$ is any finite abelian group, with Sylow subgroups $G_{p_1}, G_{p_2},\dots,G_{p_k}$, then by Theorem \ref{indecomp},
$$\Char(G) = \Char(G_{p_1}) \times \Char(G_{p_2}) \times \cdots \times \Char(G_{p_k})$$
is a decomposition of $\Char(G)$ into directly indecomposable sublattices. Suppose $G'$ is another finite abelian group, with Sylow subgroups $G_{p_1}', G_{p_2}', \dots, G_{p_l}'$, and that $\Char(G) \cong \Char(G')$. By the uniqueness of direct decompositions \cite[Corollary III.4.4]{gratzer}, we must have $k=l$ and, applying a reordering of the factors if necessary, $\Char(G_{p_i}) \cong \Char(G_{p_i}')$, for each $i\in\{1,\dots,k\}$. Thus the problem of determining when two finite abelian groups $G$ and $G'$ have isomorphic lattices of characteristic subgroups is completely reduced to the $p$-group case, i.e., we may without loss of generality assume $G$ is a $p$-group and $G'$ is a $q$-group. We then have three cases: If $p \neq 2$ and $q \neq 2$, then by Corollary \ref{corpq} we may without loss of generality assume $p=q$. This case is considered in our Main Theorem, proven below. If $p=2$ and $q \neq 2$, then $\Char(G)$ must be a distributive lattice since, by Corollary \ref{cordist}, $\Char(G')$ is; hence $G$ must have no irregular characteristic subgroups, by Theorem \ref{nondist}. The situation under which this occurs is described by Theorem \ref{noirreg} above. The last case $p=q=2$ is more complicated. We have not yet been able to obtain a complete solution for this case.

For the remainder of the paper, we consider the case $p=q\neq 2$. As usual, write
$$G=Z_{p^{\lambda_1}}\times Z_{p^{\lambda_2}}\times \cdots \times Z_{p^{\lambda_n}}.$$
By Theorem \ref{duplat}, we may without loss of generality assume that there are no duplicate factors in this decomposition of $G$, i.e. we may assume that $0 < \lambda_1 < \lambda_2 < \cdots < \lambda_n$. Likewise, as in the statement of the Main Theorem, write
$$H=Z_{p^{\mu_1}}\times Z_{p^{\mu_2}}\times \cdots \times Z_{p^{\mu_m}},$$
where, without loss of generality, we may assume $n \geq m$.

Recall that an element $x$ of a finite lattice $\mathcal{L}$ is \emph{join-irreducible} if $x$ is not the bottom element of $\mathcal{L}$ and there do not exist $y,z\in\mathcal{L}$ with $x=y\vee z$.
\begin{defn}
We denote the partially ordered set of join-irreducible elements of $\Char(G)$ by $J(G)$.
\end{defn}
Clearly, if $\Char(G)\cong\Char(H)$, then $J(G)\cong J(H)$ as partially ordered sets (and in fact, since $\Char(G)$ and $\Char(H)$ are distributive lattices, the converse of this is also true, although we will not  need to use this.) Our basic strategy for proving the Main Theorem will be to gather structural information about $J(G)$, which, it turns out, is considerably less complicated than $\Char(G)$ in certain respects. This structural information will lead us to numerical invariants on $\lambda(G)$ which will enable us to prove that $G=H$, with the exceptions stated in the theorem.
\begin{thm}\label{onedegen}
A subgroup $R(\h a)\in\Char(G)$ is a join-irreducible element of $\Char(G)$ if and only if $\h a$ has precisely one nondegenerate coordinate.
\end{thm}
\begin{proof}
Assume first that $R(\h a)\in\Char(G)$ is join-irreducible, and suppose $\h a$ has two distinct nondegenerate coordinates $i$ and $j$. Define $\h b$ by $\h b=\h a-\h e_i$, i.e. $\h b$ is the same as $\h a $ except the $i$th coordinate is decreased by one. Likewise define $\h c=\h a-\h e_j$. Then $\h b$ and $\h c$ are both canonical by Lemma \ref{degen}. We have $R(\h a)=R(\h b\vee\h c)=\langle R(\h b),R(\h c)\rangle$ where $R(\h b) \subset R(\h a)$ and $R(\h c) \subset R(\h a)$ are proper subsets, so that $R(\h a)$ is join-reducible, contrary to assumption. If, on the other hand, $\h a$ has no nondegenerate coordinates, then $a_n=a_{n-1}=a_{n-2}=\cdots=a_1=a_0=0$, i.e. $\h a=\h 0$ and $R(\h a)$ is trivial, which again contradicts the join-irreducibility of $R(\h a)$. Thus $R(\h a)$ must have precisely one nondegenerate coordinate.

Conversely, assume that $\h a$ has precisely one nondegenerate coordinate $i$. Define $\h a'=\h a-\h e_i$. Given any canonical tuple $\h b < \h a$, we claim that we must have $\h b \leq \h a'$, i.e. $b_i < a_i$; for otherwise, Theorem \ref{osplit} implies $O(\h b)=O(\h a)$, contradicting Theorem \ref{unican}. From this it follows that $R(\h a)$ is join-irreducible, for if $R(\h a)=\langle R(\h b),R(\h b')\rangle$ for proper subgroups $R(\h b)$ and $R(\h b')$ of $R(\h a)$, then $\h a=\h b\vee \h b'$ with $\h b<\h a$ and $\h b'<\h a$, hence $\h b\leq \h a'$ and $\h b'\leq \h a'$, which implies $\h b\vee\h b' \leq \h a'<\h a$, a contradiction.
\end{proof}
\begin{defn}\label{jdef}
Given $i\in\{1,\dots,n\}$ and $j\in\{1,\dots,\lambda_i\}$, we define $J(i,j)$ to be the characteristic subgroup $R(\h a)$ where
$$a_k=\begin{cases}j, & \text{if $k \geq i$}, \\ \ul{j-(\lambda_i-\lambda_k)}, & \text{if $k < i$.} \end{cases}$$
We use the notation $\ul{x}$, meaning
$$\ul{x}=\begin{cases}x, & \text{if $x\geq 0$} \\ 0, & \text{otherwise.}\end{cases}$$
We call $a_1,\dots,a_n$ the \emph{entries} of $J(i,j)$. When clarity requires us to specify the group $G$, we will write $J_G(i,j)$.
\end{defn}
\begin{example}
Let $G=Z_{p^1}\times Z_{p^3} \times Z_{p^5}\times Z_{p^7}$. Then $$J(3,3)=R(0,1,3,3).$$
\end{example}
\begin{thm}\label{jjoin}
The join-irreducible elements of $\Char(G)$ are precisely the elements $J(i,j)$.
\end{thm}
\begin{proof}
We first show that $J(i,j)=R(\h a)$ is join-irreducible. Any coordinate $k\neq i$ of $J(i,j)$ is degenerate since if $k>i$, then $a_k=a_{k-1}$, while if $k<i$, then either $a_k=0$ or $$a_{k+1}-a_k=(j-(\lambda_i-\lambda_{k+1}))-(j-(\lambda_i-\lambda_k))=\lambda_{k+1}-\lambda_k.$$
The coordinate $i$ is nondegenerate, for $j\neq 0$ and $\lambda_i\neq \lambda_{i-1}$ implies $a_{i-1}\neq a_i$, and, if $i\neq n$, then $\lambda_i\neq\lambda_{i+1}$ implies $a_{i+1}-a_i\neq\lambda_{i+1}-\lambda_i$. Thus $\h a$ has precisely one nondegenerate coordinate, namely $i$. So by Theorem \ref{onedegen}, $J(i,j)$ is join-irreducible.

Conversely, assume $R(\h a)$ is any join-irreducible element of $\Char(G)$. Then by Theorem \ref{onedegen}, $\h a$ has a unique nondegenerate coordinate $i$. The coordinates $n,n-1,\dots,i+1$ are all degenerate, and by induction, they must be degenerate of type (I). (Here we are again using the fact that the $\lambda_k$'s are all distinct.) Thus $a_n=a_{n-1}=\cdots=a_{i+1}=a_i$. Now for any $k$, $0<k<i$, the coordinate $k$ is degenerate. If it is degenerate of type (I), then by induction $k-1,k-2,\dots,1$ are also degenerate of type (I); consequently, $a_k=a_{k-1}=a_{k-2}=\cdots=a_1=a_0=0$. So either $a_k=0$ or $k$ is degenerate of type (II), i.e. $a_k=a_{k+1}-(\lambda_{k+1}-\lambda_k)$. By induction it follows that $\h a=J(i,a_i)$.
\end{proof}
The following lemma describes the partial order on $J(G)$:
\begin{lemma}\label{jpar}
Given $J(i_1,j_1), J(i_2,j_2) \in J(G)$,
$$J(i_1,j_1) \subseteq J(i_2,j_2) \iff j_2-j_1 \geq \max\{0,\lambda_{i_2}-\lambda_{i_1}\}.$$
\end{lemma}
\begin{proof}
In the case $i_1\geq i_2$, we have $\lambda_{i_2}-\lambda_{i_1}\leq 0$, so the above statement is equivalent to
$$J(i_1,j_1) \subseteq J(i_2,j_2) \iff j_1 \leq j_2.$$
Write $J(i_1,j_1)=R(\h a)$ and $J(i_2,j_2)=R(\h b)$. If $J(i_1,j_1)\subseteq J(i_2,j_2)$ then $\h a\leq\h b$, so $j_1=a_n\leq b_n=j_2$. Conversely, if $j_1 \leq j_2$, then from the definition of $J(i,j)$ it is straightforward to verify that $J(i_1,j_1) \subseteq J(i_2,j_1)\subseteq J(i_2,j_2)$.

Now consider the case $i_1\leq i_2$. Then the statement is equivalent to
$$J(i_1,j_1) \subseteq J(i_2,j_2) \iff j_2-j_1 \geq \lambda_{i_2}-\lambda_{i_1}.$$
If $J(i_1,j_1)\subseteq J(i_2,j_2)$, then we must have $$0 < j_1 = a_{i_1} \leq b_{i_1}=\ul{j_2-(\lambda_{i_2}-\lambda_{i_1})}=j_2-(\lambda_{i_2}-\lambda_{i_1}),$$
hence $j_2-j_1 \geq \lambda_{i_2}-\lambda_{i_1}$, as desired. Conversely, suppose $j_2-j_1\geq \lambda_{i_2}-\lambda_{i_1}$. Then for $k\geq i_2$ we have $a_k = j_1 \leq j_2 = b_k$. For $k \leq i_1$, since $j_1-\lambda_{i_1}\leq j_2-\lambda_{i_2}$ we have
$$a_k = \ul{j_1-(\lambda_{i_1}-\lambda_k)} \leq \ul{j_2-(\lambda_{i_2}-\lambda_k)} = b_k$$
Finally, for $i_1 \leq k \leq i_2$, since $\lambda_{i_1} \leq \lambda_{i_k}$ we have
$$a_k = j_1  \leq  j_2-(\lambda_{i_2}-\lambda_{i_1}) \leq j_2-(\lambda_{i_2}-\lambda_k)
= \ul{j_2-(\lambda_{i_2}-\lambda_k)} = b_k$$
So in every case $a_k \leq b_k$, hence $\h a \leq \h b$ and $J(i_1,j_1) \subseteq J(i_2,j_2)$.
\end{proof}
\begin{thm}\label{jdual}
The poset $J(G)$ is self-dual. An order-reversing involution is given by
$$\phi: J(i,j) \mapsto J(i,\lambda_i-j+1)$$
\end{thm}
\begin{proof}
It is clear that $\phi$ is an involution. By Lemma \ref{jpar},
\begin{align*}
& \phi(J(i_1,j_1)) \subseteq \phi(J(i_2,j_2)) \\
\iff& J(i_1,\lambda_{i_1}-j_1+1) \subseteq J(i_2,\lambda_{i_2}-j_2+1) \\
\iff& (\lambda_{i_2}-j_2+1)-(\lambda_{i_1}-j_1+1) \geq \max\{0,\lambda_{i_2}-\lambda_{i_1}\}\\
\iff& j_1-j_2 \geq \max\{\lambda_{i_1}-\lambda_{i_2},0\}\\
\iff& J(i_2,j_2) \subseteq J(i_1,j_1)
\end{align*}
so $\phi$ is order-reversing.
\end{proof}

\begin{defn}
A \emph{down set} $X$ of a poset $P$ is a subset of $P$ such that for all $x,y\in P$, if $x \leq y$ and $y\in X$ then $x\in X$. The set of all elements below a given element $p \in P$ is called a \emph{principal down set} and is denoted $p^\downarrow$.
\end{defn}

\begin{thm}\label{Jsum}
Let $J(i,j)\in J(G)$ be given. Write $J(i,j)=R(\h a)$. Then
$$|J(i,j)^\downarrow|=\sum_{k=1}^na_k.$$
\end{thm}
In other words, the number of join-irreducible subgroups of $G$ contained in $J(i,j)$ is equal to the sum of the entries of $J(i,j)$.
\begin{proof}
Write $J(i,j)=R(\h a)$. We claim that, for any $i_0$, the number of join-irreducible subgroups $J(i_0,j_0)\in J(G)$ contained in $J(i,j)$ is $a_{i_0}$. From this it clearly follows that the total number of join-irreducible subgroups contained in $J(i,j)$ is $a_1+a_2+\cdots+a_n$, as desired. So let $i_0$ be given. First suppose $i_0 \geq i$. By Lemma \ref{jpar}, we have
$$J(i_0,j_0) \subseteq J(i,j) \iff j_0 \leq j,$$
so there are $j$ suitable choices for $j_0$, namely $j_0\in\{1,\dots,j\}$. Since $a_{i_0}=j$, this proves the claim in this case. Now suppose $i_0 \leq i$. By Lemma \ref{jpar}, we have
$$J(i_0,j_0) \subseteq J(i,j) \iff j-j_0 \geq \lambda_i-\lambda_{i_0},$$
so the suitable choices for $j_0$ are $j_0\in\{1,2,\dots,j-(\lambda_i-\lambda_{i_0})\}$, a total of $\ul{j-(\lambda_i-\lambda_{i_0})}$ choices. Since $a_{i_0}=\ul{j-(\lambda_i-\lambda_{i_0})}$, this proves the claim in this case.
\end{proof}

\begin{defn}
A \emph{down-set chain} of poset $P$ is a subset of $P$ which is both a chain and a down-set of $P$. A \emph{maximal down-set chain} is a down-set chain which is not contained in any larger down-set chain.
\end{defn}
\begin{thm}\label{maxdownchain}
If $n\geq 2$ and $\lambda_n-\lambda_{n-1}\geq 2$, then $J(G)$ has precisely two maximal down-set chains, namely
\begin{align*}
D_1(G)&=\{J(i,1) : i\in\{1,\dots,n\}\},\text{ and}\\
D_2(G)&=\{J(n,j) : j\in\{1,\dots,\lambda_n-\lambda_{n-1}\}\}.
\end{align*}
\end{thm}
\begin{proof}
Note that the elements of $D_1(G)$ are $$R(1,\dots,1), R(0,1,\dots,1), R(0,0,1,\dots,1),\dots,R(0,\dots,0,1),$$
which clearly form a down-set chain. The elements of $D_2(G)$ are
$$R(0,\dots,0,1), R(0,\dots,0,2), \dots, R(0,\dots,0,\lambda_n-\lambda_{n-1})$$
which also form a down-set chain.

Now, to show that $D_1(G)$ and $D_2(G)$ are maximal, and that they are the only maximal down-set chains, we will show that any down-set chain $D$ is contained in either $D_1(G)$ or $D_2(G)$. Suppose $D$ is contained in neither $D_1(G)$ nor $D_2(G)$. Then $D$ has an element $J(i_1,j_1)$ which is not in $D_1(G)$ and an element $J(i_2,j_2)$ which is not in $D_2(G)$. Now we must have $j_1\geq 2$, so since $D$ is a down-set and $R(0,\dots,0,2)=J(n,2)\subseteq J(i_1,j_1)$, it follows that $J(n,2)\in D$. And we must have either $i_2<n$ or $j_2>\lambda_n-\lambda_{n-1}$; in either case the $(n-1)$th coordinate of $J(i_2,j_2)$ is nonzero, so $R(0,\dots,0,1,1)=J(n-1,1)\subseteq J(i_2,j_2)$, hence $J(n-1,1)\in D$. Since $J(n,2)$ and $J(n-1,1)$ are incomparable, this contradicts that $D$ is a chain.
\end{proof}

\begin{thm}\label{maxdownchain2}
If $n\geq 3$ and $\lambda_n-\lambda_{n-1}=1$, then $J(G)$ has precisely two maximal down-set chains, namely
\begin{align*}
D_1(G)=&\{J(i,1) : i\in\{1,\dots,n\}\},\text{ and}\\
D_2'(G)=&\{J(n-1,j) : j\in\{1,\dots,\lambda_{n-1}-\lambda_{n-2}\}\} \cup\\
&\{J(n,j) : j\in\{1,\dots,\lambda_{n-1}-\lambda_{n-2}+1\}\}.
\end{align*}
\end{thm}
\begin{proof}
As in the previous theorem, $D_1(G)$ is clearly down-set chain. The elements of $D_2'(G)$ are
\begin{align*}
&R(0,\dots,0,0,1), R(0,\dots,0,1,1), \\
&R(0,\dots,0,1,2), R(0,\dots,0,2,2), \\
&\dots \\
&R(0,\dots,0,\lambda_{n-1}-\lambda_{n-2}-1,\lambda_{n-1}-\lambda_{n-2}),
R(0,\dots,0,\lambda_{n-1}-\lambda_{n-2},\lambda_{n-1}-\lambda_{n-2})\\
&R(0,\dots,0,\lambda_{n-1}-\lambda_{n-2},\lambda_{n-1}-\lambda_{n-2}+1)
\end{align*}
which also form a down-set chain.

Suppose there is a down-set chain $D$ with an element $J(i_1,j_1)$ not in $D_1(G)$ and an element $J(i_2,j_2)$ not in $D_2(G)$.  We must have $j_1\geq 2$, hence $R(0,\dots,0,1,2)=J(n,2) \in D$.
If $i_2=n-1$ then $j_2>\lambda_{n-1}-\lambda_{n-2}$, while if $i_2=n$ then $j_2>\lambda_{n-1}-\lambda_{n-2}+1$; the only other possibility is $i_2\leq n-2$, so in any case the $(n-2)$th coordinate of $J(i_2,j_2)$ is nonzero, so $R(0,\dots,0,1,1,1)=J(n-2,1)\in D$. Since $J(n,2)$ and $J(n-2,1)$ are incomparable, this contradicts that $D$ is a chain.
\end{proof}

\begin{proof}[Proof of Main Theorem]
The ``if" part has already been shown in Theorems \ref{thmchain} and \ref{z25z124}. So assume $\Char(G) \cong \Char(H)$, with $\phi$ a lattice isomorphism mapping $G$ onto $H$. We may assume $n\geq 2$ and $m \geq 2$, since if either $n<2$ or $m<2$, then $\Char(G)$ and $\Char(H)$ are chains, and this case has been completely characterized by Theorem \ref{thmchain}.

One simple observation is that $\Char(G)$ and $\Char(H)$ must have the same number of join-irreducible elements, which by Theorem \ref{jjoin} implies
\begin{equation}\label{eqsum}\sum_{k=1}^n\lambda_k=\sum_{k=1}^n\mu_k.\end{equation}
In other words, $G$ and $H$ must have the same order.

For $1 \leq i \leq n$, we define
\begin{align*}
\alpha_i&=\lambda_{n-i+1}-\lambda_{n-i},\\
\beta_i&=\mu_{m-i+1}-\mu_{m-i},
\end{align*}
The proof splits into two cases, according to whether $\alpha_1>1$ or $\alpha_1=1$:\\
\emph{Case 1:} First assume $\alpha_1>1$, i.e., $\lambda_n-\lambda_{n-1}>1$. In this case we claim that $\mu_m-\mu_{m-1}>1$ also. For note that the elements
\begin{align*}
J_G(n-1,1)&=R(0,\dots,0,1,1)\\
J_G(n,2)&=R(0,\dots,0,0,2)
\end{align*}
are distinct elements of $J(G)$ with their entries summing to 2, and that these are the only such elements of $J(G)$. Since $J(G) \cong J(H)$, Theorem \ref{Jsum} implies that $J(H)$ must also have two such elements; as the only candidates are $J_H(m-1,1)$ and $J_H(m,2)$, we must have $\mu_m-\mu_{m-1}>1$ since otherwise we would have
$$J_H(m,2)=R(0,\dots,0,1,2),$$
whose entries sum to 3 instead of 2. Note that a similar argument shows that if $\beta_1>1$ then $\alpha_1>1$, so in fact $\alpha_1=1$ if and only if $\beta_1=1$. (We will use this below in Case 2.)

Now, by Theorem \ref{maxdownchain}, $J(G)$ has precisely two maximal down set chains $D_1(G)$ and $D_2(G)$; similarly $J(H)$ has precisely two maximal down set chains $D_1(H)$ and $D_2(H)$. Since a lattice isomorphism clearly maps maximal down set chains to maximal down set chains, we must have either $\phi(D_1(G))=D_1(H)$ or $\phi(D_1(G))=D_2(H)$.

First assume $\phi(D_1(G))=D_1(H)$. Then, $n=|D_1(G)|=|D_1(H)|=m$, and for all $0 \leq i\leq n-1$ we have
$$\phi(J_G(n-i,1))=J_H(n-i,1).$$ Hence, by Theorem \ref{jdual},
\begin{align*}
|J_G(n-i,\lambda_{n-i})^\downarrow |=|J_G(n-i,1)^\uparrow |
=|J_H(n-i,1)^\uparrow |=|J_H(n-i,\lambda_{n-i})^\downarrow|.
\end{align*}
By Theorem \ref{Jsum} and Definition \ref{jdef}, this implies
$$\sum_{k=1}^{n-i}\ul{\lambda_{n-i}-(\lambda_{n-i}-\lambda_k)} + i\lambda_{n-i}
=\sum_{k=1}^{n-i}\ul{\mu_{n-i}-(\mu_{n-i}-\mu_k)} + i\mu_{n-i}, $$
i.e.,
\begin{equation}\label{sum1}\sum_{k=1}^{n-i}\lambda_k + i\lambda_{n-i}
=\sum_{k=1}^{n-i}\mu_k + i\mu_{n-i}.
\end{equation}
We note that this also holds (trivially) for $i=n$ since $\lambda_0=\mu_0=0$. Using (\ref{eqsum}), we may rewrite (\ref{sum1}) as
$$\sum_{k=n-i+1}^n\lambda_k-i\lambda_{n-i}=\sum_{k=n-i+1}^n\mu_k-i\mu_{n-i},$$
which is equivalent to
$$\sum_{k=1}^ik\alpha_k = \sum_{k=1}^ik\beta_k.$$
From this it follows easily by induction that $\alpha_i=\beta_i$ for all $1 \leq i \leq n$, which implies $\lambda_i=\mu_i$ for all $1 \leq i \leq n$, so $G=H$ and we are done.

So assume instead that $\phi(D_1(G))=D_2(H)$. In this case, $\phi(D_2(G))=D_1(H)$, and so $\alpha_1=\lambda_n-\lambda_{n-1}=|D_2(G)|=|D_1(H)|=m$, and likewise $\beta_1=\lambda_m-\lambda_{m-1}=|D_2(H)|=|D_1(G)|=n$. For all $0 \leq i \leq n-1$ we have $$\phi(J_G(n-i,1))=J_H(m,i+1).$$
Hence, by Theorem \ref{jdual},
\begin{align*}
|J_G(n-i,\lambda_{n-i})^\downarrow |=|J_G(n-i,1)^\uparrow |
=|J_H(m,i+1)^\uparrow |=|J_H(m,\mu_m-i)^\downarrow |.
\end{align*}
By Theorem \ref{Jsum} and Definition \ref{jdef}, this implies
$$\sum_{k=1}^{n-i}\ul{\lambda_{n-i}-(\lambda_{n-i}-\lambda_k)} + i\lambda_{n-i}
=\sum_{k=1}^m\ul{\mu_m-i-(\mu_m-\mu_k)}, $$
i.e.,
$$\sum_{k=1}^{n-i}\lambda_k + i\lambda_{n-i}
=\sum_{k=1}^m\ul{\mu_k-i}. $$
If we define
$$\epsilon_i=\sum_{k=1}^m(\ul{\mu_k-i}-(\mu_k-i)),$$
then this becomes
$$\sum_{k=1}^{n-i}\lambda_k + i\lambda_{n-i}
=\sum_{k=1}^m(\mu_k-i)+\epsilon_i = \sum_{k=1}^m\mu_k-im+\epsilon_i $$
Applying (\ref{eqsum}), this becomes
$$\sum_{k=n-i+1}^n\lambda_k-i\lambda_{n-i}=im-\epsilon_i,$$
which is equivalent to
\begin{equation}\label{sum2}\sum_{k=1}^ik\alpha_k=im-\epsilon_i.\end{equation}
Now, we claim that for $1 \leq i \leq n-1$,
\begin{equation}\label{eqn3}i\alpha_i=m-(\epsilon_i-\epsilon_{i-1}).\end{equation}
This holds for $i=1$ since we know $\alpha_1=m$ and $\epsilon_0=\epsilon_1=0$. By induction, (\ref{sum2}) gives
\begin{align*}
i\alpha_i&=im-\sum_{k=1}^{i-1}k\alpha_k-\epsilon_i\\
&=im-\sum_{k=1}^{i-1}(m-(\epsilon_k-\epsilon_{k-1}))-\epsilon_i = m-(\epsilon_i-\epsilon_{i-1}),
\end{align*}
proving the claim.

Recall we are assuming $m\leq n$. Consider the case $m<n$. Here we may take $i=m$ in (\ref{eqn3}), giving
$$m\alpha_m=m-(\epsilon_m-\epsilon_{m-1})$$
Observe that
\begin{equation}\label{eqneps}\epsilon_i-\epsilon_{i-1}=|\{k : \mu_k<i\}|\end{equation}
Hence $\epsilon_m-\epsilon_{m-1}\geq 0$. Since $\alpha_m>0$, this forces $\epsilon_m-\epsilon_{m-1}=0$ and $\alpha_m=1$. It follows from (\ref{eqneps}) that  $\epsilon_i-\epsilon_{i-1}=0$ for all $i\leq m$. Hence, since $\epsilon_0=0$, we must have $\epsilon_i=0$ for all $i\leq m$. Equation (\ref{eqn3}) then becomes
$$i\alpha_i=m$$
for $1\leq i\leq m$. So $m$ is divisible by every positive integer less than it. This implies $m=2$. If we had $n\geq 4$, then we could take $i=3$ in (\ref{eqn3}), giving
$$3\alpha_3=2-(\epsilon_3-\epsilon_2),$$
which is a contradiction since $\alpha_3>0$ and $\epsilon_3-\epsilon_2\geq 0$. Hence we must have $n=3$. So we have $\alpha_1=m=2$, $\alpha_2=\alpha_m=1$, and $\beta_1=n=3$. Hence,
\begin{align*}
\lambda(G)&=(\lambda_1,\lambda_1+1,\lambda_1+3)\\
\lambda(H)&=(\mu_1,\mu_1+3)
\end{align*}
Now, note that $J_G(3,3)=R(0,1,3)$ and $J_H(1,2)=R(2,2)$ are the only elements in $J(G)$ and $J(H)$ respectively with entries summing to 4. It follows that $\phi(J_G(3,3))=J_H(1,2)$. By Theorem \ref{jdual}, it follows that
$$|J_G(3,\lambda_1+1)^\downarrow |=|J_G(3,3)^\uparrow |=|J_H(1,2)^\uparrow |=| J_H(1,\mu_1-1)^\downarrow|.$$
Now we have
\begin{align*}
J_G(3,\lambda_1+2)&=R(\ul{\lambda_1-2},\lambda_1-1,\lambda_1+1)\\
J_H(1,\mu_1-1)&=R(\mu_1-1,\mu_1-1)
\end{align*}
So by Theorem \ref{Jsum},
$$\ul{\lambda_1-2}+(\lambda_1-1)+(\lambda_1+1)=(\mu_1-1)+(\mu_1-1)$$
Now, if $\lambda_1>1$, then this would imply
$$3\lambda_1-2=2\mu_1-2,$$
while (\ref{eqsum}) gives
$$3\lambda_1+4=2\mu_1+3,$$
a contradiction. So we must have $\lambda_1=1$, from which it follows that
\begin{align*}
\lambda(G)&=(1,2,4)\\
\lambda(H)&=(2,5),
\end{align*}
so $G=Z_p\times Z_{p^2}\times Z_{p^4}$ and $H=Z_{p^2}\times Z_{p^5}$.

It remains to consider the case $m=n$. First suppose $n\geq 5$. Then taking $i=n-1$ in (\ref{eqn3}) gives
$$(n-1)\alpha_{n-1}=n-(\epsilon_{n-1}-\epsilon_{n-2}).$$
Since $n\geq 5$ implies $2(n-1)>n$, we must have $\alpha_{n-1}=1$ and $\epsilon_{n-1}-\epsilon_{n-2}=1$. Now, taking $i=n-2$ in (\ref{eqn3}) gives
$$(n-2)\alpha_{n-2}=n-(\epsilon_{n-2}-\epsilon_{n-3}).$$
Since $n\geq 5$ implies $2(n-2)>n$, we must have $\alpha_{n-1}=1$ and $\epsilon_{n-2}-\epsilon_{n-3}=2$. But this contradicts that from (\ref{eqneps}), $\epsilon_i-\epsilon_{i-1}$ is an increasing sequence.

Now consider the case $m=n=4$. Taking $i=3$ in (\ref{eqn3}) gives
$$3\alpha_3=4-(\epsilon_3-\epsilon_2),$$
which forces $\alpha_3=1$ and $\epsilon_3-\epsilon_2=1$. Taking $i=2$ in (\ref{eqn3}) gives
$$2\alpha_2=4-(\epsilon_2-\epsilon_1).$$
Since it is impossible to have $\epsilon_2-\epsilon_1=2$, this forces $\alpha_2=2$ and $\epsilon_2-\epsilon_1=0$. We also know $\alpha_1=m=4$. By a similar argument, $\beta_1=4, \beta_2=2, \beta_3=1$. It follows from (\ref{eqsum}) that $\lambda_i=\mu_i$ for all $i$, so $G=H$.

Now consider the case $m=n=3$. Taking $i=2$ in (\ref{eqn3}) gives
$$2\alpha_2=3-(\epsilon_2-\epsilon_1),$$
which forces $\alpha_2=1$. Since $\alpha_1=m=3$ and similarly $\beta_1=3, \beta_2=1$, we obtain $G=H$.

Finally consider the case $m=n=2$. Then we have $\alpha_1=\beta_1=n=2$, and again, $G=H$.\\
\emph{Case 2:} Now assume $\lambda_1=1$. In this case, we also have $\mu_1=1$. Here we may assume $n \geq 3$ and $m \geq 3$, since otherwise $\Char(G)$ or $\Char(H)$ would be a chain by Theorem \ref{thmchain}. So, by Theorem \ref{maxdownchain2}, $J(G)$ has precisely two maximal down set chains $D_1(G)$ and $D_2'(G)$; similarly $J(H)$ has precisely two maximal down set chains $D_1(H)$ and $D_2'(H)$. If $\phi(D_1(G))=D_1(H)$, then, as in Case 1, we obtain $G=H$.

So we may assume $\phi(D_2'(G))=D_1(H)$. It follows that $m=|D_1(H)|=|D_2'(G)|=2(\lambda_{n-1}-\lambda_{n-2})+1=2\alpha_2+1$ This gives
\begin{align*}
\phi(J_G(n-2i,1))&=J_H(m,i+1), \quad\text{for $0 \leq i < \frac{n}2$},\\
\phi(J_G(n-2i-1,1))&=J_H(m-1,i+1), \quad\text{for $0 \leq i < \frac{n-1}2$}.
\end{align*}
 From Theorem \ref{jdual}, it follows that
\begin{align*}
|J_G(n-2i,\lambda_{n-2i})^\downarrow |&=|J_H(m,\mu_m-i)^\downarrow |,\\
|J_G(n-2i-1,\lambda_{n-2i-1})^\downarrow |&=|J_H(m-1,\mu_{m-1}-i)^\downarrow |.
\end{align*}
Applying Theorem \ref{Jsum} yields
\begin{align}
\label{eqn4}\sum_{k=1}^{n-2i}\lambda_k+2i\lambda_{n-2i} &= \sum_{i=1}^m\ul{\mu_k-i}\\
\label{eqn5}\sum_{k=1}^{n-2i-1}\lambda_k+(2i+1)\lambda_{n-2i-1} &= \sum_{k=1}^{m-1}\ul{\mu_k-i}+(\mu_{m-1}-i).
\end{align}
Now,
$$\sum_{k=1}^m\ul{\mu_k-i}=\sum_{k=1}^m(\mu_k-i)+\epsilon_i=\sum_{k=1}^m\mu_k-im+\epsilon_i.$$
Putting this with (\ref{eqn4}) and applying (\ref{eqsum}), this gives
$$\sum_{k=n-2i+1}^n\lambda_k-2i\lambda_{n-2i}=im-\epsilon_i,$$
which is equivalent to
\begin{equation}\label{eqn6}\sum_{k=1}^{2i}k\alpha_k=im-\epsilon_i\end{equation}
On the other hand, in (\ref{eqn5}),
\begin{align*}
\sum_{k=1}^{m-1}\ul{\mu_k-i}+(\mu_{m-1}-i)
&= \sum_{k=1}^m\ul{\mu_k-i}-(\mu_m-i)+(\mu_{m-1}-i)\\
&= \sum_{k=1}^m(\mu_k-i)-(\mu_m-\mu_{m-1})+\epsilon_i\\
&= \sum_{k=1}^m\mu_k-im-1+\epsilon_i
\end{align*}
Putting this with (\ref{eqn5}) and applying (\ref{eqsum}), this gives
$$\sum_{k=n-2i}^n\lambda_k-(2i+1)\lambda_{n-2i-1}=im+1-\epsilon_i,$$
which is equivalent to
\begin{equation}\label{eqn7}
\sum_{k=1}^{2i+1}k\alpha_k=im+1-\epsilon_1.
\end{equation}
Assume for the moment that $n\geq 4$. Then we may apply (\ref{eqn6}) and (\ref{eqn7}) with $i=1$, giving
\begin{align*}
\alpha_1+2\alpha_2&=m\\
\alpha_1+2\alpha_2+3\alpha_3&=m+1.
\end{align*}
Subtracting these yields
$$3\alpha_3=1,$$
which is a contradiction. So it only remains to consider the case $n=m=3$. In this case, we may still apply (\ref{eqn6}) with $i=1$, yielding
$$\alpha_1+2\alpha_2=3.$$
Since $\alpha_1=1$, this gives $\alpha_2=1$. Similarly, $\beta_1=\beta_2=1$. From (\ref{eqsum}), we then obtain $\lambda_i=\mu_i$ for all $i$, hence $G=H$.
\end{proof}

\bibliography{characteristic_subgroups}

\end{document}